\documentclass[11pt]{article}

\usepackage{amsmath, amssymb, amsfonts, amsthm, graphicx, xcolor}
\usepackage{authblk}
\usepackage[margin=3.0cm]{geometry}
\usepackage{epstopdf}
\usepackage{hyperref}
\newtheorem{theorem}{\bf Theorem}[section]

\newtheorem{proposition}{\bf Proposition}[section]

\newtheorem{lemma}{\bf Lemma}[section]
\newtheorem{remark}{\bf Remark}[section]
\providecommand{\keywords}[1]{\textbf{\textit{Keywords: }} #1}

\newcommand{\rev}[1]{\textcolor{black}{#1}}
\def\RR{\mathbb R}
\def\EE{\mathbb E}

\def\DH{{\mathcal{D}}}
\def\LL{L^1(\DH)}

\def\LH{L^1(\DH)}

\def\LHBi{L^1(\DH;L^2(\Omega))}

\def\f{u}

\def\e{\varepsilon}

\def\var{{\rm Var}}
\def\cov{{\rm Cov}}

\def\be{\begin{equation}}
\def\ee{\end{equation}}
\def\bea{\begin{eqnarray}}
\def\eea{\end{eqnarray}}

\begin{document}

\title{Multi-Order Monte Carlo IMEX hierarchies for uncertainty quantification in multiscale hyperbolic systems}

\author[$\dagger$]{Giulia Bertaglia \footnote{Corresponding author. Email address: \textit{giulia.bertaglia@unife.it}}}
\author[$\mathsection$,$\ddagger$]{Walter Boscheri}
\author[$\star$,$\ddagger$]{Lorenzo Pareschi}

\affil[$\dagger$]{\small Department of Environmental and Prevention Sciences, University of Ferrara, Ferrara, Italy}
\affil[$\mathsection$]{\small Laboratoire de Mathématiques, Université Savoie Mont Blanc, CNRS, Chambéry, France}
\affil[$\star$]{\small Maxwell Institute and Department of Mathematics, Heriot-Watt University, Edinburgh, UK}
\affil[$\ddagger$]{Department of Mathematics and Computer Science, University of Ferrara, Ferrara, Italy}

\date{June 28, 2026}
\maketitle

\begin{abstract}
We introduce a novel Multi-Order Monte Carlo approach for uncertainty quantification in the context of multiscale time-dependent partial differential equations. The new framework leverages Implicit-Explicit Runge–Kutta time integrators to satisfy the asymptotic-preserving property across different discretization orders of accuracy. In contrast to traditional Multi-Level Monte Carlo methods, which require costly hierarchical re-meshing, our method constructs a multi-order hierarchy by varying both spatial and temporal discretization orders within the Monte Carlo framework. This enables efficient variance reduction while naturally adapting to the multiple scales inherent in the problem ensuring asymptotic consistency.
The proposed method is particularly well-suited for hyperbolic systems with stiff relaxation, kinetic equations, and low Mach number flows, where standard Multi-Level Monte Carlo techniques often encounter computational challenges. Numerical experiments demonstrate that the novel Multi-Order Monte Carlo approach achieves substantial reduction of both error and variance while maintaining asymptotic consistency in the asymptotic limit.
\end{abstract}

\keywords{Uncertainty Quantification, Multi-Order Monte Carlo, multiscale hyperbolic problems, Asymptotic-Preserving schemes, IMEX Runge–Kutta schemes}

\tableofcontents

\section{Introduction}
Scientific computing has become a central tool across many research fields for analyzing the dynamics of complex systems. This is especially true in engineering and biological applications, where experimental studies are often time-consuming, costly, and difficult to replicate.
The primary goal of numerical simulations is to predict physical phenomena governed by mathematical models, typically formulated as ordinary differential equations (ODEs) or, more commonly, partial differential equations (PDEs). These models rely on simplifying assumptions that preserve their validity while enabling tractable representations of complex behaviors.
Numerous deterministic numerical methods have been developed to solve such models efficiently, with well-characterized and controllable numerical errors.
However, these numerical techniques generally assume precise knowledge of the initial and boundary conditions, the computational domain, and all model parameters. In practice, this assumption is rarely met due to limitations in measurement capabilities, observational biases, and incomplete understanding of the underlying processes.
Some uncertainties, known as epistemic uncertainties, stem from incomplete knowledge of the system and can be reduced through better measurements or more advanced modeling techniques. Others, known as aleatoric or random uncertainties, are intrinsic to the system and may be irreducible.
Therefore, any realistic numerical approach must account for these sources of uncertainty, while also recognizing the limitations inherent in both the models and the numerical methods used to solve them.

A vast body of literature is devoted to uncertainty quantification (UQ) techniques for systems governed by ODEs and PDEs. In recent years, this area has seen significant growth \cite{xiu2010,jin2017,pareschi2020}, largely driven by three main factors:
(i) the increased availability of data due to technological advances;
(ii) the rapid development of high-performance computing tools;
(iii) and the construction of increasingly sophisticated numerical algorithms.
Among the available UQ methods, Monte Carlo (MC) and Multi-Level Monte Carlo (MLMC) approaches have gained widespread popularity due to their robustness, non-intrusiveness, and suitability for parallel computing, even in high-dimensional settings \cite{caflisch1998,heinrich2001,giles2008,barth2011,menhorn2024}. Their non-intrusiveness is especially attractive in the context of non-linear hyperbolic systems, where intrusive methods like the stochastic Galerkin approach often compromise critical properties such as hyperbolicity and positivity of the solution \cite{pareschi2020}.

Despite their robustness, MC methods are known for their slow convergence, which can lead to prohibitively high computational costs, particularly when each model evaluation involves the solution of a complex PDE. The MLMC method, initially proposed in \cite{heinrich2001} and further developed in \cite{giles2008,mishra2012,mishra2012a,sukys2013}, mitigates this drawback by combining estimators across multiple levels of mesh resolution using a telescopic sum. This hierarchical structure significantly reduces the computational burden compared to standard MC methods, and it easily adapts to the concept of multi-fidelity methods \cite{gorodetsky2020generalized,zanoni2024improved,geraci2017multifidelity}, where different mathematical models of increasing complexity can be used to control and possibly reduce the variance, hence accelerating the convergence of MC methods.

To further improve efficiency, several variance reduction techniques have been developed based on multi-order discretizations, that provide a different approach toward multi-fidelity methods. These approaches aim to balance computational cost and accuracy by organizing solution hierarchies according to polynomial degree rather than mesh refinement.
A notable example is the Multi-order Monte Carlo method introduced by Motamed and Appel\"o~\cite{motamed2018}, which employs a discontinuous Galerkin framework for UQ in hyperbolic PDEs with stochastic inputs. Their method constructs a polynomial hierarchy on a fixed mesh, reducing variance without the need for re-meshing.
Another approach is the $p$-refined Monte Carlo method by Blondeel et al.~\cite{blondeel2020}, which uses polynomial refinement within Galerkin Finite Element methods to build an efficient quasi-Monte Carlo variance reduction scheme. This strategy has shown strong performance in engineering applications.

In the field of Stochastic Collocation, Beck et al. \cite{beck2019} proposed a multi-index stochastic collocation method that extends classical sparse grids techniques by incorporating multiple discretization orders. Their method has been successfully applied to UQ in random PDEs defined on arbitrary domains.
Additionally, Grote et al. \cite{grote2022} introduced a multi-order time discretization strategy using local time-stepping for stochastic wave propagation problems, demonstrating how multi-order methods can reduce computational costs while maintaining accuracy. Other relevant contributions include Motamed’s work on a multi-order neural network surrogate model for PDE-based UQ \cite{motamed2020}, and the multi-order Monte Carlo strategies for computational fluid dynamics developed by Dürrwächter et al.~\cite{durrwachter2023}. Therefore, overall, multi-order approaches present a compelling alternative to traditional MLMC and multi-fidelity techniques, especially in scenarios where re-meshing is computationally expensive or where polynomial-based hierarchies naturally integrate with the discretization framework.

In this work, we present and analyze a novel extension of the Multi-Order Monte Carlo (MOMC) methodology tailored to multiscale problems which ensures asymptotic consistency in the stiff relaxation limits of the governing equations. We introduce AP-MOMC, the first extension of MOMC to the setting of Asymptotic-Preserving (AP) schemes for multiscale hyperbolic PDEs. This variant exploits the efficiency and flexibility of Implicit–Explicit (IMEX) Runge–Kutta methods \cite{IMEXbook,Pareschi2005}, enabling simulations of hyperbolic systems with uncertain inputs while retaining asymptotic consistency and avoiding the drawbacks of intrusive techniques. In this framework, following a bi-fidelity approach, the physical scalings and asymptotic limits of the model can be systematically incorporated as additional variance-reduction levels, thereby enriching the MOMC hierarchy and further improving efficiency. The proposed MOMC framework builds a {weighted} combination of numerical solutions obtained with different discretization orders, in the spirit of multi-fidelity Monte Carlo methods \cite{geraci2015,peherstorfer2016,peherstorfer2018,dimarco2019,dimarco2021,pareschi2022,iacomini2025,lin2025}, that are proven to be asymptotic preserving. Unlike MLMC, which requires nested meshes and a telescopic structure, MOMC admits arbitrary discretizations, making it particularly effective for complex geometries. The MOMC estimator can also be interpreted as a generalized control variate strategy, with weights chosen to minimize variance. 

The rest of this paper is structured as follows. Section 2 introduces the general problem setting and presents motivating examples. Section 3 defines the MOMC estimators and the extension to the AP and bi-fidelity setting. In Section 4 the statistical properties of the method are analyzed. Numerical results are reported in Section 5, and conclusions are drawn in Section 6.

\section{Hyperbolic relaxation systems with uncertain inputs}
We introduce the mathematical problem we are interested in considering a probabilistic framework that employs the concept of random variable in probability theory to randomize a set of deterministic parameters to real-valued independent random variables $z_1,\ldots,z_{d_z}$. All these random inputs can be collected in a single vector $z=(z_1,\ldots,z_{d_z})^T\in \Omega\subset\mathbb{R}^{d_z}$ defined in a complete probability space $(\Omega,\mathcal{A},\mathcal{P})$ consisting of a set of outcomes $\Omega$, the $\sigma$-algebra of events $\mathcal{A}$, and probability measure $\mathcal{P}$ \cite{chen2013,bertaglia2021a}. Each uncertain input $z_i$, $i=1,\ldots,d_z$, is characterized by its own probability density function (PDF) $P_i(z_i)$ and, since we assume that the random inputs identify a set of mutually independent random variables, the PDF of the random vector $z$ is defined as $P(z)=\prod_{i=1}^{d_z}P_i(z_i)$ \cite{xiu2010}.

In this framework, the solution of the problem will not only depend on the physical variables space $x\in \mathcal{D}\subset\RR^{d_x}$, $d_x \geq 1$, and time $t\in \mathbb{R}_0^{+}$, but also on the uncertain input vector $z$. 
For instance, let us consider the Jin-Xin system as prototype one-dimensional (1D) multiscale hyperbolic model with relaxation \cite{jin1995a,bertaglia2023a}:
\begin{subequations}
	\begin{align}
		\frac{\partial u}{\partial t} + \frac{\partial v}{\partial x} &= 0\,, \label{eq:JX1}\\
		\frac{\partial v}{\partial t} + a^2\frac{\partial u}{\partial x}  &= -\frac1{\e}\left( v-F(u)\right)\,,
		\label{eq:JX2}
	\end{align}
\end{subequations}
where $u=u(x,t,z)$ and $v=v(x,t,z)$, $\e$ is a small positive parameter called \emph{relaxation rate} or \emph{scaling parameter} (related to the mean free path of particles in kinetic theory \cite{pareschi2013}), and $a$ is a positive constant. If $a$ satisfies the sub-characteristic condition, $a^2 > F'(u)^2$ \cite{jin1995a}, for $\e \ll 1$, the above system is a good approximation of the conservation law
\be
\frac{\partial u}{\partial t} + \frac{\partial F(u)}{\partial x}  = 0\,.
\label{eq:cons}
\ee
Indeed, in the small relaxation limit, i.e., for $\e \to 0^+$, we get the \emph{local equilibrium} $v = F(u)$, which leads to \eqref{eq:cons}.

Notice that system \eqref{eq:JX1}-\eqref{eq:JX2} is coupled with appropriate initial conditions $u(x,0,z) = u_0(x,z)$, $v(x,0,z) = v_0(x,z) = F\left(u_0(x,z)\right)$ and boundary conditions $\mathcal{B}_{u,\partial \mathcal{D}}\left(u,x,t,z\right) = g_u(t,z)$, $\mathcal{B}_{v,\partial \mathcal{D}}\left(v,x,t,z\right) = g_v(t,z)$, 
that might also depend on the uncertain input vector, as well as modeling parameters in general.

Hyperbolic relaxation systems with uncertainty have been extensively investigated, for example in the pioneering works of Jin et al. \cite{Jin2015,Jin2016}, where stochastic Galerkin methods were used to quantify uncertainty in this class of problems. These studies highlighted the dual challenge of (i) handling stiffness in $\e$, which makes direct simulation costly, and (ii) ensuring consistency with asymptotic limits in the presence of random inputs. While intrusive polynomial chaos methods provide accuracy, they become expensive in high-dimensional settings and are difficult to couple with multiscale solvers.

In this work, instead, we adopt a non-intrusive Monte Carlo framework and propose to enhance it with multi-order variance reduction strategies combined with Asymptotic-Preserving (AP) IMEX Runge–Kutta finite volume schemes \cite{IMEXbook,Pareschi2005} of varying spatial and temporal orders of accuracy. These schemes ensure uniform accuracy across all $\e$-regimes while retaining consistency with the asymptotic limits of the model, , thereby satisfying the AP property (i.e., consistency with the asymptotic limits of the mathematical model even at the discrete level), which is of utmost importance when dealing with multiscale models \cite{jin2022}. Crucially, realizations of the random vector $z$ may correspond to very different stiffness regimes, so the AP property guarantees robust performance across the whole probability space. While a realization of the PDE system corresponds to the solution obtained for a specific choice of the random parameters, our goal is instead to compute statistical information about the quantities of interest, such as their expected values and variances.

\section{Multi-Order Monte Carlo IMEX methods}
If $z\in \Omega\subseteq \RR^{d_z}$, $d_z \geq 1$, is distributed as $P(z)$, we denote the expected value of the variable $u$ by
\be
\EE[u](x,t,z) = \int_{\Omega} u(x,t,z)P(z)\,dz.
\label{eq:E}
\ee
Furthermore, given $M$ i.i.d. samples of $z$, $z_1,\ldots , z_M$, we recall the definition of the Monte Carlo (MC) estimator for a generic random variable $u(x,t,z)$ \cite{caflisch1998}:
\be
E_M\left[u\right](x,t) = \frac1{M}\sum_{k=1}^Mu(x,t,z_k).
\label{eq:MC}
\ee

Let us assume that for the computation of each $k$-th MC realization we can choose between two distinct numerical approximations of the problem of different order of accuracy, such that one constitutes our high-accurate solution of order $L$, $u_L(x,t,z)$, while the other is the (less computationally expensive) lower accurate solution of order $(L-1)$, $u_{L-1}(x,t,z)$. Then,
the order $L$ solver is evaluated in $M_L$ random samples, while 
the order $(L-1)$ solver is evaluated in 
\begin{equation}
\label{eq:samples}
M_{L-1}=M_L(1+r) \gg M_L
\end{equation}
random samples, with $r \in \mathbb{N}^+$ representing the factor of additional realizations.
Notice that, in practice, $r$ can be estimated considering the computational cost of the two numerical solvers, being $r\approx \mathcal{C}_L/\mathcal{C}_{L-1}-1$, with $\mathcal{C}_l$ computational cost of the $l$-th order solution, \rev{following a cost-based robust allocation inspired by multifidelity model management.}

\begin{remark}
\rev{It is worth noting that, in the context of the present work, we deal with stiff multiscale PDEs whose numerical performance is assessed through spatial error indicators across different asymptotic regimes. In this setting, a fully optimized sample allocation would require estimating variances and correlations that depend not only on the quantity of interest, but also on the norm used to measure the error, the time, and the multiscale regime under consideration. Furthermore, in strongly correlated multi-order hierarchies, the covariance-dependent weights introduced below, which determine the optimized allocation, may be highly sensitive to finite-sample estimation errors. For these reasons, rather than adopting the optimization strategy proposed in \cite{peherstorfer2016,peherstorfer2018}, we employ a simpler cost-equivalent sample allocation, which was found to be more robust in the finite-sample regimes considered here. A detailed investigation of optimal sample allocation strategies, tailored to the specific error norm used to assess the method's performance, is left for future work.}
\end{remark}

Assuming to ensure the correlation across levels (shared seeds and independence between sample groups), we can then approximate the stochastic variable with the following parameter-dependent control-variate approach \cite{dimarco2019}:
\begin{equation}
	u(x,t,z) \approx u^{MOMC}(x,t,z) = u_{L,M_L}(x,t,z) -\alpha_L\left( u_{L-1,M_L}(x,t,z) - u_{L-1,M_{L-1}}(x,t,z)\right),
	\label{eq:uMOMC}
\end{equation}
where $\alpha_L \in \mathbb{R}$ and $u_{l,M_l}$ represents the solution of $u$ obtained using the $l$-th accurate solver evaluated for $M_l$ Monte Carlo samples.
Therefore, the Multi-Order Monte Carlo estimator reads
\begin{equation}
	\mathbb{E}[u](x,t) \approx E_{M_L}\left[u_L\right](x,t) -\alpha_L\left( E_{M_L}\left[u_{L-1}\right](x,t) - E_{M_{L-1}}\left[u_{L-1}\right](x,t)\right),
	\label{eq:BOMC}
\end{equation}
where we have assumed that $\mathbb{E}[u_{L-1}](x,t)\approx E_{M_{L-1}}\left[u_{L-1}\right](x,t)$.

Notice that this is an unbiased estimator of the solution, indeed when $\alpha_L = 0$ we recover exactly the standard Monte Carlo estimator for the sole $L$-th order solution.
At the same time, the computation of the variance (coincident with the mean squared error) of the control-variate approximation follows with the Lemma below.
\begin{lemma}
	The variance of the MOMC control-variate approximation \eqref{eq:uMOMC} is
	\begin{equation}
		\var\left[u^{MOMC}\right] =\var\left[u_{L,M_L} \right] +\frac{r}{(1+r)}\left(\alpha_L^2\, \var\left[u_{L-1,M_L}\right] - 2\alpha_L \,\cov\left[ u_{L,M_L}, u_{L-1,M_L}\right]\right).
		\label{eq:var}
	\end{equation}
\end{lemma}

\begin{proof}
	Evaluating the variance of \eqref{eq:uMOMC}, we have
	\begin{subequations}
		\begin{align*}
			\var \left[u^{MOMC}\right] &= \var\left[u_{L,M_L}\right]+\alpha_L^2\,\var\left[\left( u_{L-1,M_L} - u_{L-1,M_{L-1}}\right)\right]\\
			&- 2\alpha_L \,\cov\left[ u_{L,M_L}, \left( u_{L-1,M_{L-1}} - u_{L-1,M_{L-1}}\right)\right]\\
			&= \var\left[u_{L,M_L}\right]+\alpha_L^2\left( \var\left[u_{L-1,M_L}\right] +  \var\left[u_{L-1,M_{L-1}}\right] \right) -2\alpha_L^2\,\cov\left[ u_{L-1,M_L}, u_{L-1,M_{L-1}}\right]\\
			&-2\alpha_L\,\left(\cov\left[ u_{L,M_L}, u_{L-1,M_L} \right] - \cov\left[ u_{L,M_L}, u_{L-1,M_{L-1}}\right]\right).
		\end{align*}	
	\end{subequations}
	Since, following \cite{peherstorfer2016},
	\begin{equation}
		\var\left[u_{l,M}\right]=\frac1{M}\var\left[u_l\right],
		\label{def.VarM}
	\end{equation}
	\begin{equation}
		\cov\left[u_{l,M},u_{h,N}\right]=\frac1{\mathrm{max}\{M,N\}}\cov\left[u_l,u_l\right],
	\end{equation}
	we further obtain
	\begin{equation*}
		\var [u^{MOMC}] = \var\left[u_{L,M_L}\right]+\alpha_L^2\left( \var\left[u_{L-1,M_L}\right] -  \var\left[u_{L-1,M_{L-1}}\right] \right) -2\alpha_L\,\left( \frac1{M_L} - \frac1{M_{L-1}}\right) \cov\left[ u_L, u_{L-1} \right],
	\end{equation*}
	which, recalling that $M_{L-1}=M_L(1+r)$, leads to \eqref{eq:var}.
\end{proof}

It is now possible to compute the optimal value of $\alpha_L$ that minimizes the variance \eqref{eq:var} of the approximate solution.
\begin{theorem}
	If $\var\left[u_{L-1,M_L}\right]\ne0$, the quantity
	\begin{equation}
		\alpha_L^* = \frac{\cov\left[ u_{L,M_L}, u_{L-1,M_L}\right]}{\var\left[u_{L-1,M_L}\right]}
		\label{alpha_opt}
	\end{equation}
	minimizes the variance of $u^{MOMC}$ at the point $(x,t)$ and gives
	\begin{equation}
		\var\left[u^{MOMC}\right] =\left( 1-\frac{r}{1+r}\rho_{L,L-1}^2\right) \var\left[u_{L,M_L}\right],
		\label{var_opt}
	\end{equation}
	where $\rho_{L,L-1} \in [-1,1]$ is the Pearson correlation coefficient between the $L$-th and the $(L-1)$-th accurate solutions, i.e.,
	\begin{equation*}
		\rho_{L,L-1}^2 = \frac{\cov^2\left[ u_L, u_{L-1}\right]}{\var\left[u_L\right]\var\left[u_{L-1}\right]}.
	\end{equation*}
\end{theorem}
\begin{proof}
	By direct differentiation of \eqref{eq:var} with respect to $\alpha_L$ we find that $\alpha_L^*$ in \eqref{alpha_opt} is the unique stationary point. The fact that $\alpha_L^*$ is a minimum follows from the positivity of the second derivative:
	\[ \frac{2r}{1+r} \var\left[u_{L-1,M_L}\right] >0.\]
	Then, substituting \eqref{alpha_opt} in \eqref{eq:var} we obtain
	\begin{equation*}
		\var\left[u^{MOMC}\right] =\var\left[u_{L,M_L}\right] -\frac{r}{(1+r)}\frac{\cov^2\left[ u_{L,M_L}, u_{L-1,M_L}\right]}{\var\left[u_{L-1,M_L}\right]},
	\end{equation*}
	which leads to \eqref{var_opt}.
\end{proof}
Notice that from \eqref{alpha_opt} we can also write
\[
\alpha_L^* = \rho_{L,L-1}\var\left[u_{L,M_L}\right]^{1/2} \,.
\]
As a consequence, we have that $\alpha_L^* \in \left[- \var\left[u_{L,M_L}\right]^{1/2},\var\left[u_{L,M_L}\right]^{1/2}\right]$.

\subsection{Higher order MOMC hierarchies}
We can easily extend the treatment to additional orders of discretization.
To estimate $\EE[u_{L-1}]$ we use $M_{L-1}\gg M_L$ samples and consider $u_{L-2}$ as control variate:
\[
\EE[u_{L-1}]\approx E_{M_{L-1}}[u_{L-1}]-\alpha_{L-1}\left(E_{M_{L-1}}[u_{L-2}]-\EE[u_{L-2}]\,.\right).
\]
Similarly, in a recursive way, we can construct estimators for the remaining expectations of the control variate $\EE[u_{L-2}],\EE[u_{L-3}],\ldots,\EE[u_2],\EE[u_1]$ using, respectively, $M_{L-2}\ll M_{L-3}\ll\ldots \ll M_2\ll M_1$ samples until 
\[
\EE[u_2]\approx E_{M_{2}}[u_2]-\alpha_{2}\left(E_{M_{2}}[u_{1}]-\EE[u_{1}]\right),
\]
where we stop with the final estimate 
\[
\EE[u_1]\approx E_{M_1}[u_1].
\]
By combining the estimators of each stage together, we obtain the recursive MOMC estimator:
\begin{eqnarray}
	\nonumber
	E_L^{\alpha}[u] &=& E_{M_L}[u_L] -\alpha_L\left(E_{M_L}[u_{L-1}]-E_{M_{L-1}}[u_{L-1}]\right.\\
	\nonumber
	&+&\alpha_{L-1}\left(E_{M_{L-1}}[u_{L-2}]-E_{M_{L-2}}[u_{L-2}]\right.\\[-.3cm]
	\label{eq:MOMC}
	\\[-.2cm]
	\nonumber
	&\ldots&\\
	\nonumber
	&+&\alpha_{2}\left.\left.\left(E_{M_{2}}[u_{1}]-E_{M_{1}}[u_{1}]\right) \ldots\right)\right),
\end{eqnarray}
where we denote with $\alpha=(\alpha_2,\ldots,\alpha_L)^T$ the vector of control parameters.

Now, if we compute the optimal values $\alpha_l^*$ independently for each recursive stage, by ignoring the errors due to the approximations of the various expectations, we obtain again \eqref{alpha_opt}.
We refer to this set of values, which avoids the solution of the linear system resulting from the consideration of all the errors of the approximations of all the expectations, as \emph{quasi-optimal}.

Note that, since the control variate $u_l$ and $u_{l-1}$ are known on the same set of samples $M_l$, the values $\alpha^*_l$ can be estimated using the following unbiased estimators:
\bea
\label{eq:varm}
{\var}_{M_l}[u_{l-1}] &=& \frac1{M_l-1}\sum_{k=1}^{M_l} \left(u_{l-1}-E_{M_l}[u_{l-1}]\right)^2,\\
\label{eq:covm}
{\cov}_{M_l}[u_l,u_{l-1}] &=& \frac1{M_l-1}\sum_{k=1}^{M_l} \left(u_l-E_{M_l}[u_l]\right)\left(u_{l-1}-E_{M_l}[u_{l-1}]\right),
\eea
which allow to compute
\be
\label{alpha_qopt}
{\alpha}_{l}^*= \frac{{\cov}_{M_l}[u_l,u_{l-1}]}{{\var}_{M_l}[u_{l-1}] },
\ee 
and therefore ${\alpha}_{l}^* \in [-{\var}_{M_l}[u_l]^{1/2},{\var}_{M_l}[u_l]^{1/2}]$.

\subsection{Asymptotic-Preserving MOMC and bi-fidelity enhancements}
\label{AP-MOMC}
A key limitation of standard MOMC is that, while it reduces variance efficiently, it does not automatically guarantee consistency with the asymptotic limits of multiscale PDEs. This is crucial in relaxation systems, where the $\varepsilon \to 0$ limit corresponds to a reduced model. To address this, we consider an Asymptotic-Preserving Multi-Order Monte Carlo (AP-MOMC) framework, which design MOMC using AP-IMEX Runge–Kutta schemes \cite{IMEXbook}. This ensures that variance reduction and asymptotic consistency at all scales are simultaneously achieved, even at the discrete level.

Beyond this, the AP-MOMC hierarchy can be further enriched by introducing bi-fidelity levels that exploit asymptotic scalings of the underlying model. For the Jin–Xin system, this means incorporating the local equilibrium \eqref{eq:cons} as the coarsest level of the hierarchy. This equilibrium can be numerically discretized using again a suitable finite volume scheme of the lowest desired order of accuracy.
The resulting bi-fidelity AP-MOMC approach preserves the key asymptotic properties of the mathematical model at the discrete level, while simultaneously benefiting from the model’s inherent scalings to enrich the hierarchical structure of the numerical method.

Therefore, the bi-fidelity AP-MOMC estimator reads
\begin{eqnarray}
	\nonumber
	E_{\tilde L}^{\alpha}[u] &=& E_{M_L}[u_L] -\alpha_L\left(E_{M_L}[u_{L-1}]-E_{M_{L-1}}[u_{L-1}]\right.\\
	\nonumber
	&+&\alpha_{L-1}\left(E_{M_{L-1}}[u_{L-2}]-E_{M_{L-2}}[u_{L-2}]\right.\\[-.3cm]
	\label{eq:APMOMC}
	\\[-.2cm]
	\nonumber
	&\ldots&\\
	\nonumber
	&+&\alpha_{2}\left(E_{M_{2}}[u_{1}]-E_{M_{1}}[u_{1}]\right.\\
	\nonumber
	&+&\alpha_{1}\left.\left.\left(E_{M_{1}}[\tilde u_{1}]-E_{M_{0}}[\tilde u_{1}]\right) \ldots\right)\right),
\end{eqnarray}
where $\tilde u_1$ identifies the solution of the asymptotic scaling obtained with the method of order 1, while $\alpha_1$ is the additional control parameter depending on the covariance between the two distinct models (full-order and reduced-order). Clearly, the reduced-order model (being computationally cheaper than the full-order model due to the smaller number of equations to solve) will be evaluated using a significantly larger number of samples, i.e., $M_0 \gg M_1$.

Notice that, for $\epsilon \to 0^+$, due to the AP property of the schemes used, we have that $E_{M_{l}}[u_{l}] \to E_{M_{l}}[\tilde u_{l}]$, $l=1,\ldots,L$, and $\alpha_1=1$. Indeed, in the asymptotic limit, the full-order coincides with the reduced-order model. Consequently, the expectation of the variable $u$ substantially reduces to \eqref{eq:MOMC} applied to the reduced-order model, except for the increased number of samples $M_0$ we can afford to use for the lowest order of the hierarchy:
\begin{eqnarray}
	\nonumber
	E_{\tilde L,0}^{\alpha}[u] &=& E_{M_L}[\tilde u_L] -\alpha_L\left(E_{M_L}[\tilde u_{L-1}]-E_{M_{L-1}}[\tilde u_{L-1}]\right.\\
	\nonumber
	&+&\alpha_{L-1}\left(E_{M_{L-1}}[\tilde u_{L-2}]-E_{M_{L-2}}[\tilde u_{L-2}]\right.\\[-.3cm]\nonumber
	\\[-.2cm]
	\nonumber
	&\ldots&\\
	\nonumber
	&+&\alpha_{2}\left.\left.\left(E_{M_{2}}[\tilde u_{1}]-E_{M_{0}}[\tilde u_{1}]\right) \ldots\right)\right).
\end{eqnarray}

A schematic representation of the AP-MOMC hierarchical structures is given in Figure \ref{APMOMC_scheme}.

\begin{figure}[htb]
	\centering
	\includegraphics[width=0.95\textwidth]{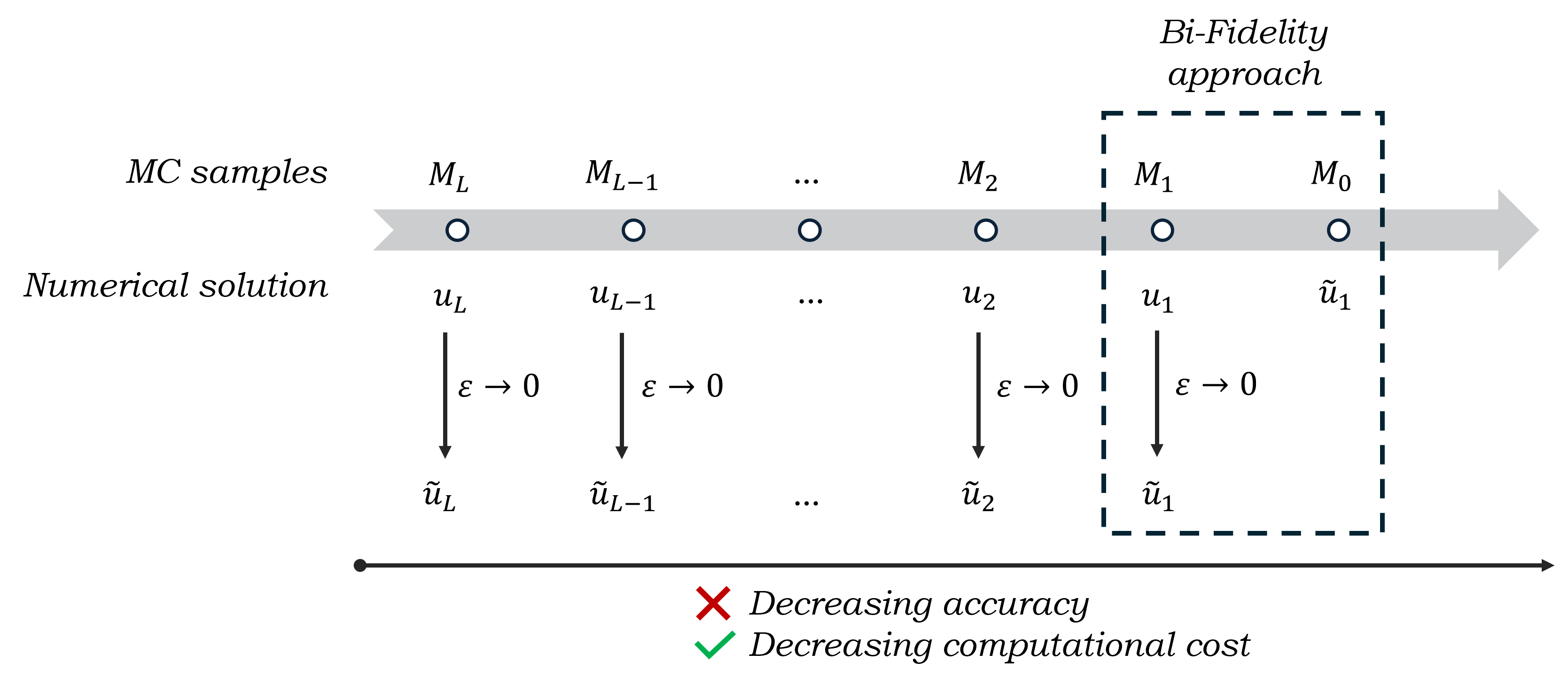}
	\caption{Schematic representation of the hierarchical structure for AP-MOMC and bi-fidelity AP-MOMC. The hierarchy is built from solutions obtained with AP-IMEX finite volume methods of different accuracy orders ($u_l$ denotes the solution of order $l$). In the bi-fidelity variant, an additional level $\tilde{u}_1$ is included, corresponding to the solution of the local equilibrium (asymptotic limit) computed with the lowest-order method, thus incorporating a bi-fidelity correction. The diagram emphasizes that both approaches remain consistent with the asymptotic limits of the multiscale dynamics, where $\varepsilon$ is the scaling parameter.}
	\label{APMOMC_scheme}
\end{figure}

\section{Error bounds}
Before going into the details of error bounds, we introduce some notations and assumptions that will be used in the sequel.

For a random variable $Z$ taking values in ${L^p(\DH)}$, we define 
\be
\|Z\|_{{L^p(\DH;L^2(\Omega))}}=\|\EE[Z^2]^{1/2}\|_{{L^p(\DH)}}.
\label{eq:norm1}
\ee
The above norm, if $p\neq 2$, differs from   
\be
\|Z\|_{L^2(\Omega;{L^p(\DH))}}=\EE\left[\|Z\|^2_{L^p(\DH)}\right]^{1/2},
\label{eq:norm2}
\ee
used for example in \cite{mishra2012a}. Note that by Jensen inequality \cite{rudin1987} we have
\be
\|Z\|_{{L^p(\DH;L^2(\Omega))}} \leq \|Z\|_{L^2(\Omega;{L^p(\DH))}}.
\ee
To avoid unnecessary difficulties, in the sequel we consider norm \eqref{eq:norm1} for $p=1$. The same results hold true for $p=2$ (the two norms coincides) whereas their extension to norm \eqref{eq:norm2} for $p=1$ typically requires $Z$ to be compactly supported. We refer to \cite{mishra2012a,dimarco2019} and \cite[Chapter 7]{jin2017} for further details.

Let us also assume that the deterministic solver for \eqref{eq:JX1}-\eqref{eq:JX2}, if the initial data $u_0$ and $v_0$ are sufficiently regular, satisfies the estimate (see \cite{dimarco2014,dimarco2019,mouhot2006,russo2012})
\be
\|u(\cdot,t^n)-u^n_l\|_{{\LL}} \leq C_l \left(\Delta x^l + \Delta t^l \right),
\label{eq:det}
\ee
where $C_l$ is a positive constant which depends on time and on the initial data, and $u^n_l$ is the computed approximation of the deterministic solution $u(x,t)$ at time $t^n$ on the computational mesh having $\Delta x$ as spatial cell size and $\Delta t$ as time step size. Again, the positive integer $l\ge1$ characterizes the accuracy of the discretizations in space-time. 

We emphasize that the estimates here presented are purposely of a general nature to illustrate the characteristics of the method, being aware that the application to specific models can clearly lead to more targeted estimates but it is outside the objectives of the present work.  

Let us now consider the error bound that we obtain using \eqref{eq:MOMC} with the sub-optimal values given by \eqref{alpha_qopt}. In the sequel, we shall discuss only the consistency error. 
We observe that if, at each level, we denote the estimator
\[
{E}^{\alpha^*_l}_{l}[\f^n_l] = E_{M_l}[\f^n_l]-{\alpha_l^{*,n}}\left(E_{M_l}[\f^n_{l-1}]-\EE[\f^n_{l-1}]\right),
\]
then, by ignoring the statistical error in estimating ${\alpha_l^{*,n}}$, we have \cite{caflisch1998}
\[
\|\EE[\f_l^n]-{E}^{\alpha_l^*}_{l}[\f^n_l]\|_{{\LHBi}} \leq  \tilde C_l\sigma^n_l
M_l^{-1/2},
\]
where $\tilde C_l>0$ is a suitable constant and we defined
\be
\sigma^n_l=\left\|\left(1-(\rho^n_{l,l-1})^2)\right)^{1/2}\var(\f^n_l)^{1/2}\right\|_{\LH}.
\ee
Using the recursive estimator, in combination with a deterministic solver satisfying \eqref{eq:det}, we can compute
\bea
\nonumber
&& \|\EE[\f](\cdot,t^n)-{E}^{\alpha^*}_{L}[\f^n]\|_{{\LHBi}} \\
\nonumber
&& \hskip 4cm \leq  
\|\EE[\f](\cdot,t^n)-\EE[\f^n_L]\|_{{\LH}} +
\|\EE[\f^n_L]-{E}^{\alpha^*}_{L}[\f^n]\|_{{\LHBi}}.
\eea 
The first term is bounded by the discretization error of the higher order scheme
\[
\begin{split}
	\|\EE[\f](\cdot,t^n)-\EE[\f^n_L]\|_{{\LH}} \leq  C_L\left(\Delta x^{L}+\Delta t^{L}\right) ,
\end{split}
\]
whereas, ignoring the statistical errors in estimating the quasi-optimal vector of values $\alpha^*_l$, the second term can be estimated recursively as
\bea
\nonumber
\|\EE[\f^n_L]-{E}^{\alpha^*}_{L}[\f^n]\|_{{\LHBi}}
&\leq& 
\|\EE[\f^n_L]-{E}^{\alpha^*_L}_{L}[\f^n_L]\|_{{\LHBi}}+\|\alpha_L^*({\EE}[\f^n_{L-1}]-{E}^{\alpha^*}_{L-1}[\f^n])\|_{{\LHBi}}\\
\nonumber
&\leq& 
\tilde C_L 
\left\{\sigma^n_L M_L^{-1/2}+
\tau^n_L\|\EE[\f^n_{L-1}]-{E}^{\alpha^*}_{L-1}[\f^n]\|_{{\LHBi}}\right\}
\\
\nonumber
& \leq& \tilde C_{L}
\left\{\sigma^n_L M_L^{-1/2}+\tau^n_L \tilde C_{L-1}\left\{\sigma^n_{L-1} M_{L-1}^{-1/2}\right.\right.\\
\label{eq:errREC}
&& + \left.\left. \|\alpha_{L-1}^*(\EE[\f^n_{L-2}]-{E}^{\alpha^*}_{L-2}[\f_{L-2}^n])\|_{{\LHBi}}\right\}\right\}
\\
\nonumber
&& \ldots\\
\nonumber
& \leq& \tilde C \left( \sum_{l=1}^L \xi^n_l \sigma^n_l M_l^{-1/2}\right),
\eea
where we defined $\xi^n_L=1$, $\sigma^n_1=\left\|\var[\f^n_1]^{1/2}\right\|_{\LH}$,
and, for $l=1,\ldots, L-1$,
\be
\xi^n_l = \prod_{j=l+1}^L \tau^n_j,\quad 
\tau^n_l=\left\|\rho^n_{l,l-1}\var[\f^n_l]^{1/2}\right\|_{\LH}.
\label{eq:muh}
\ee
Thus we have proved the following result.
\begin{proposition}
	Consider a deterministic scheme which satisfies the order assumptions \eqref{eq:det} for the solution of \eqref{eq:JX1}-\eqref{eq:JX2} with random initial data $u(x,0,z) = u_0(x,z)$, $v(x,0,z) = v_0(x,z) = F\left(u_0(x,z)\right)$. Assume that the initial data is sufficiently regular.
	Then, the recursive MOMC estimate defined in \eqref{eq:MOMC} satisfies the error bound 
	\be
	\|\EE[\f](\cdot,t^n)-{E}^{\alpha_l^*}_{L}[\f^n_L]\|_{{\LHBi}} \leq \tilde C \left(\sum_{l=1}^L \xi_l \sigma_l M_l^{-1/2}\right)+C_L\left(\Delta x^{L}+\Delta t^{L}\right)\,,
	\label{eq:errREC2}
	\ee
	where $\xi_l$ are given by \eqref{eq:muh} and
	$\tilde C>0$ depends on the final time and on the initial data. 
\end{proposition}

\begin{remark}
	If we assume a standard MOMC estimator, which corresponds to the choice $\alpha_l=1$, $l=1,\ldots,L$, we get the estimate
	\be
	\|\EE[\f](\cdot,t^n)-{E}^{\bf{1}}_{L}[\f^n_L]\|_{{\LHBi}}  \leq \tilde C \left(\sum_{l=1}^L \xi_l \sigma_l M_l^{-1/2}\right)+C_L\left(\Delta x^{L}+\Delta t^{L}\right).
	\label{eq:errREC2b}
	\ee
\end{remark}

\begin{remark}
	The statistical estimates in this work are expressed using the norm  \eqref{eq:norm1} with $p=1$, i.e.
	$\| \cdot \|_{L^1(\DH; L^2(\Omega))}$, to accommodate possible low regularity in space, as encountered in multiscale hyperbolic problems.
	However, under suitable assumptions on the regularity of the solution, an alternative formulation using \eqref{eq:norm2} with $p=2$, i.e.
	$\| \cdot \|_{L^2(\Omega; L^2(\DH))}$, leads to simpler expressions for the MOMC variance,
	closely resembling the standard MLMC setting \cite{mishra2012,mishra2012a}. In particular, the statistical error admits a bound of the form
	\be
	\|\EE[\f^n_L]-{E}^{\alpha^*}_{L}[\f^n_L]\|_{L^2(\Omega; L^2(\DH))} \leq \hat C \left(\sum_{l=1}^L \xi_l \sigma_l M_l^{-1/2}\right).
	\label{eq:errREC3}
	\ee
\end{remark}

\section{Numerical tests}
In this section, we demonstrate the effectiveness of the proposed MOMC, AP-MOMC, and bi-fidelity AP-MOMC methods through numerical experiments on three different test cases: the inviscid Burgers' equation, the shallow water equations, and a multiscale blood flow model. The underlying numerical schemes are based on finite volume discretizations combined with the method-of-lines time integrators.

\subsection{Inviscid Burgers' equation}
Before addressing multiscale cases, to illustrate the MOMC methodology in a simplified setting, we first consider the inviscid Burgers’ equation, which is a standard prototype for nonlinear conservation laws:
\be
\nonumber
\frac{\partial u}{\partial t} + u\frac{\partial u}{\partial x} = 0\,.
\ee
We consider a Gaussian initial distribution of the state variable $u$, 
\be
\nonumber
u(x,0,z) = \frac1{\sqrt{2\pi}\sigma(z)} e^{-\frac{x^2}{2\sigma^2(z)}},
\ee
with an uncertain input $z$ affecting its variance, being $\sigma(z)=z$, $z \sim \mathcal{U}(-1,1)$. We set the physical domain $\mathcal{D}=[-5,5]$ and the final time of the simulation $t_{end} = 2.5$.

Since the equation does not present stiff terms, we solve the problem applying the proposed MOMC approach using, as deterministic solvers, the following three explicit numerical methods with increasing order of accuracy.

\begin{enumerate}
	\item Explicit Runge–Kutta scheme of order $L=1$ (corresponding to the Forward Euler) coupled with a first order Finite Volume space discretization and Godunov numerical fluxes. The computational cost of a single run of the first order solver is $C_1 = 1\times1=1$.
	
	\item Explicit Runge–Kutta scheme of order $L=2$ (corresponding to the $2^{nd}$ order Heun’s scheme) coupled with a second order Finite Volume space discretization based on a TVD reconstruction with minmod slope limiter and Godunov numerical fluxes. The computational cost of a single run of the second order solver is $C_2 = 2\times2=4$.
	
	\item Explicit Runge–Kutta scheme of order $L=3$ (corresponding to the $3^{rd}$ order Heun’s scheme) coupled with a third order Finite Volume discretization in space based on a WENO3 reconstruction of the boundary values and Godunov numerical fluxes. The computational cost of a single run of the third order solver is $C_3 = 3\times3=9$.
\end{enumerate} 

The time step size is limited by the classical $\mathrm{CFL}$ stability condition, being $\mathrm{CFL} = 0.9$, and we discretize the physical domain with 200 computational cells.

In Figure \ref{test_burg_sol} we show a reference solution computed with the standard MC method with the explicit Runge–Kutta Finite Volume solver of order 3 using $M=20000$ stochastic samples.

\begin{figure}[b!]
	\centering
	\includegraphics[width=0.48\textwidth]{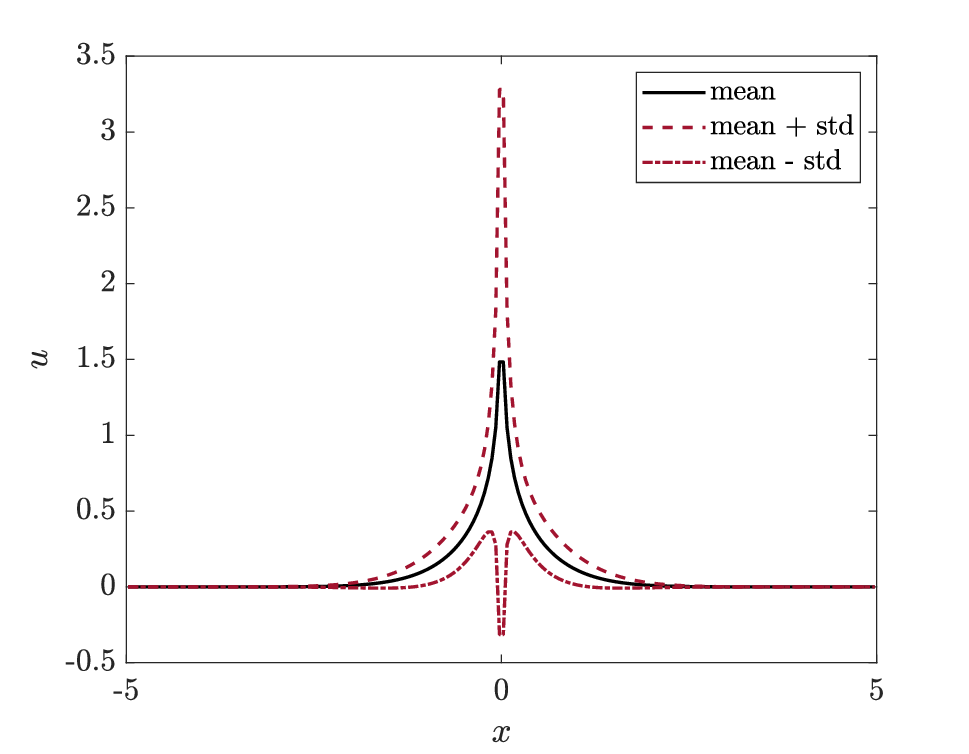}
	\includegraphics[width=0.48\textwidth]{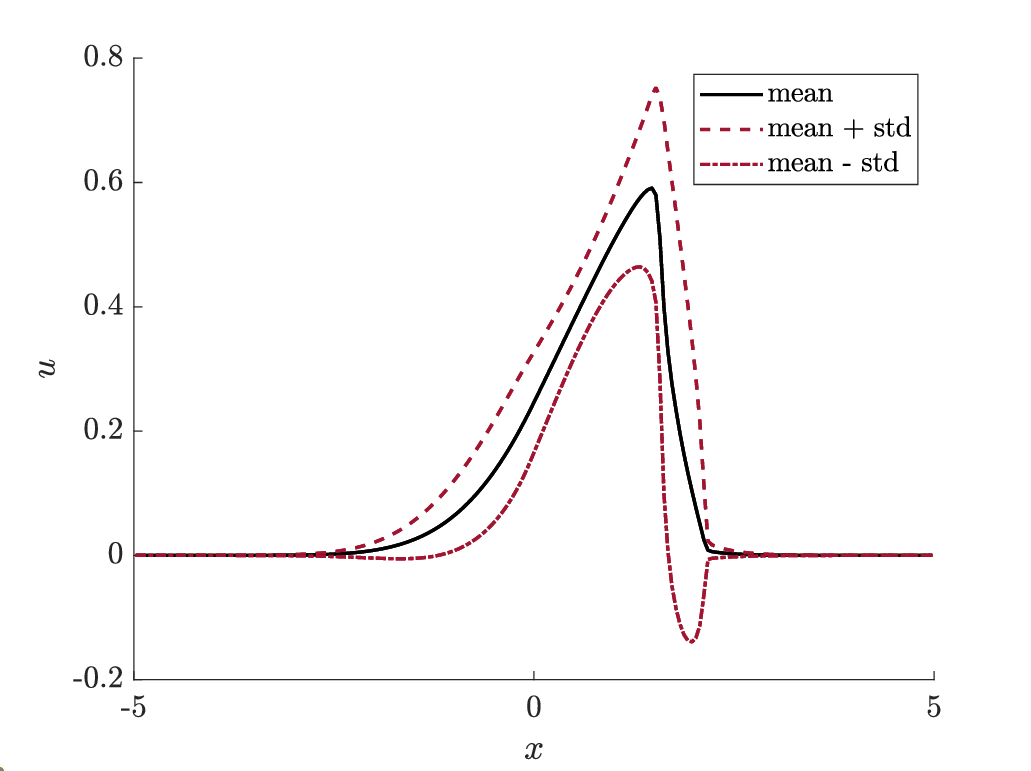}
	\caption{Burgers equation. Reference solution at the initial (left) and final (right) time of the simulation.}
	\label{test_burg_sol}
\end{figure}

\begin{figure}[th!]
	\centering
	\includegraphics[width=0.48\textwidth]{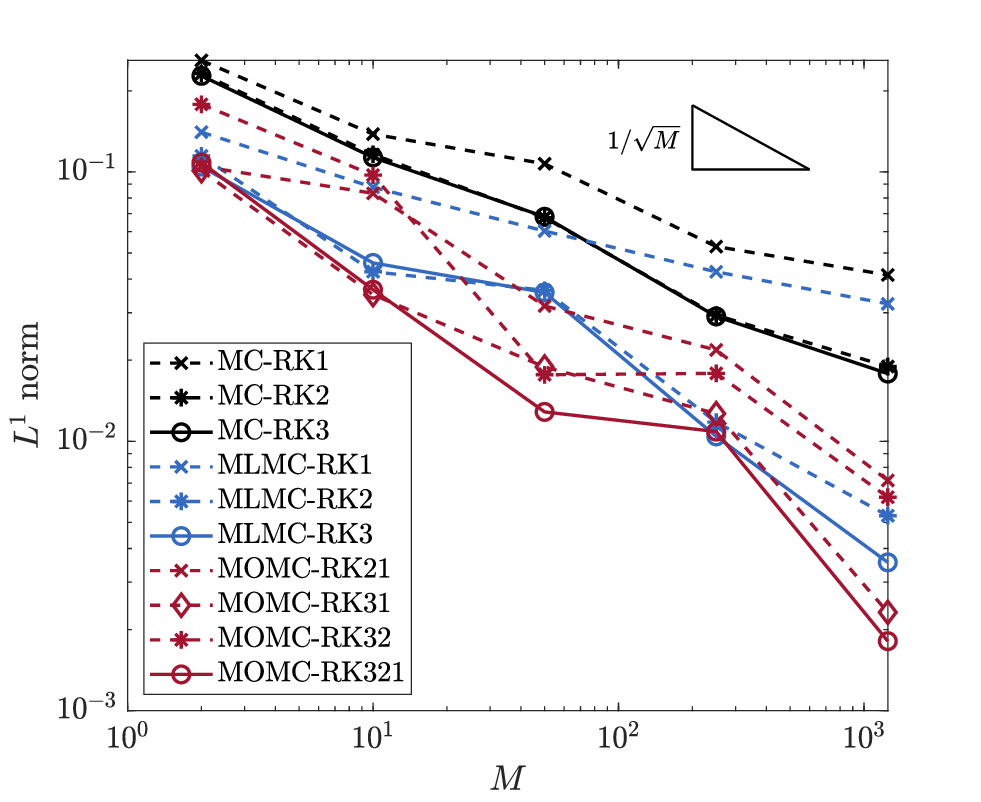}
	\includegraphics[width=0.48\textwidth]{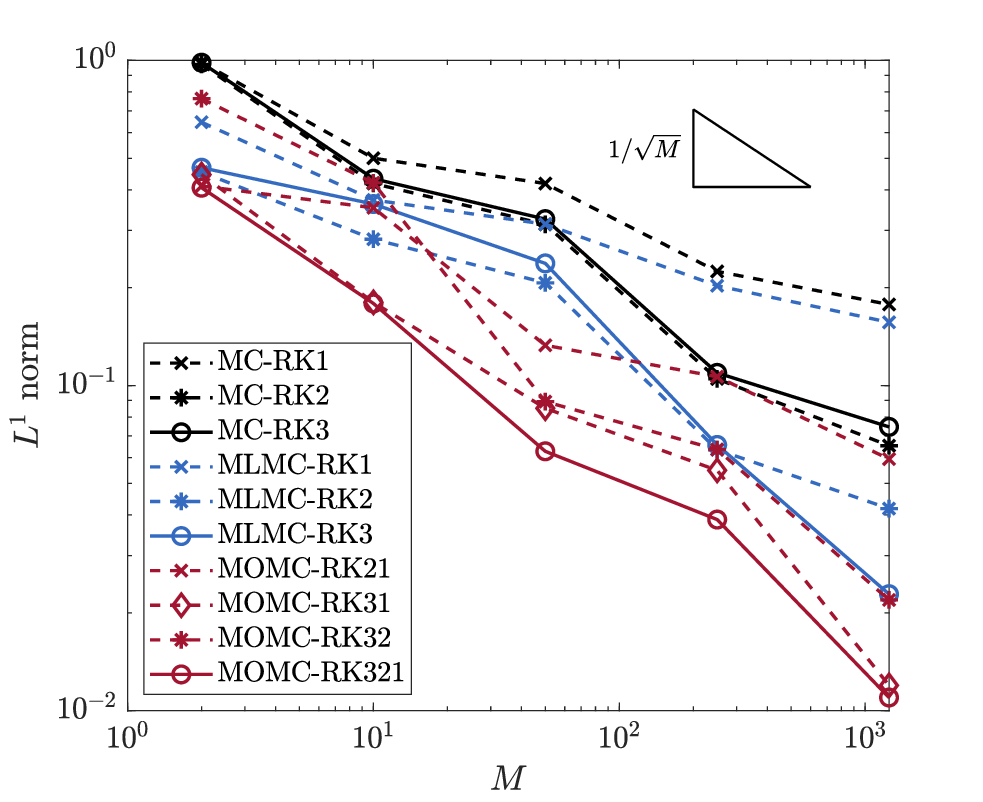}
	\caption{Burgers equation. $L^1$-norm error in the expectation (left) and in the variance (right) of the variable $u$ with respect to the number of samples $M$ used in the standard Monte Carlo (MC) methods and in the $L$-th level of the different Multi-Level MC (MLMC) and Multi-Order MC (MOMC) methods used.}
	\label{test_burg_M}
\end{figure}
\begin{figure}[th!]
	\centering
	\includegraphics[width=0.7\textwidth]{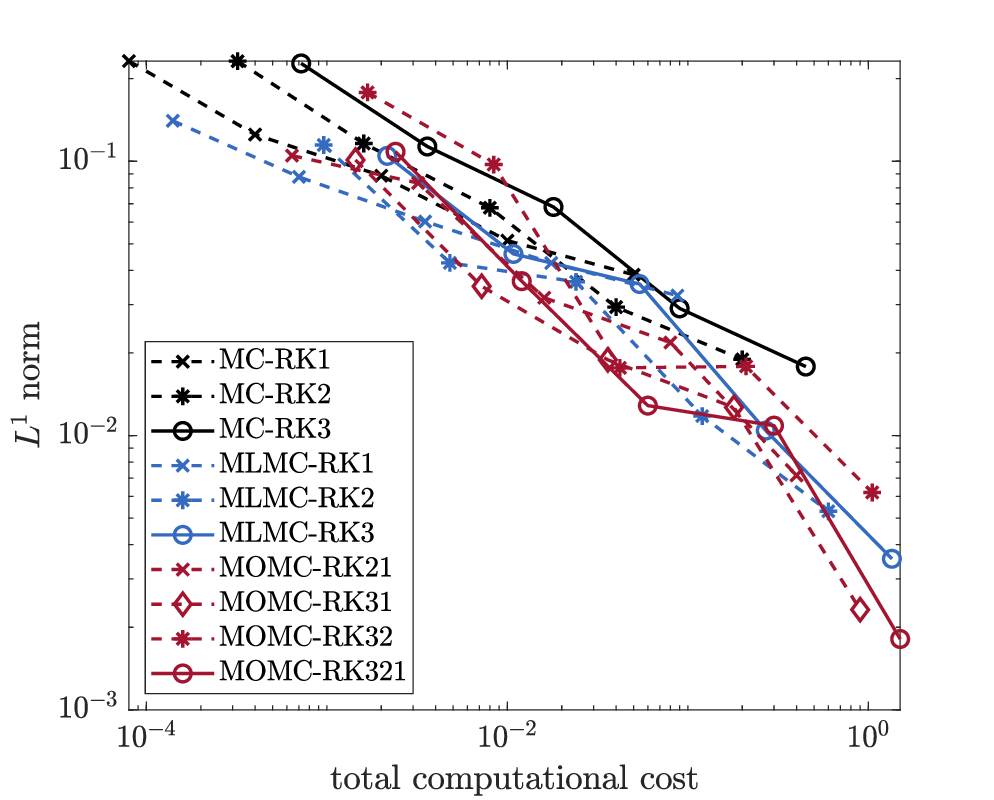}
	\caption{Burgers equation. $L^1$-norm error in the expectation of the variable $u$ with respect to the total computational cost for the different Monte Carlo (MC), Multi-Level MC (MLMC) and Multi-Order MC (MOMC) methods used.}
	\label{test_burg_tCPU}
\end{figure}

We compute the $L^1$-norm convergence curves of the solution using various versions of the proposed MOMC approach. Specifically, we consider two-level combinations of orders 2 and 1 (MOMC-RK21), orders 3 and 1 (MOMC-RK31), and orders 3 and 2 (MOMC-RK32), as well as a three-level configuration using orders 3, 2, and 1 (MOMC-RK321) of the deterministic algorithm.
We compare the MOMC results with convergence curves obtained using the standard Monte Carlo method with order 3 (MC-RK3), order 2 (MC-RK2), and order 1 (MC-RK1) schemes for each realization. Additionally, we include comparisons with a three-level Multi-Level Monte Carlo method, using order 3 (MLMC-RK3), order 2 (MLMC-RK2), and order 1 (MLMC-RK1) schemes for each realization, employing a mesh coarsening corresponding to 100 and 50 cells on the $(L-1)$ and $(L-2)$ levels, respectively. 

In Figure \ref{test_burg_M}, we present the $L^1$-norm convergence curves for the expectation and variance of the solution with respect to the number of samples $M$ used in the standard MC methods and at the $L$-th level of the different MLMC and MOMC methods. We observe that both MLMC and MOMC techniques consistently outperform the standard MC methods of the corresponding deterministic order of accuracy. Moreover, the proposed MOMC methods exhibit a clear reduction in the expectation error for the same $M$ at the highest level, along with a noticeable variance reduction.

Figure \ref{test_burg_tCPU} shows the convergence curves with respect to the total computational cost (accounting for the cumulative solver runs at the various levels) for each MC, MLMC, and MOMC configuration. This plot further validates the effectiveness of the proposed methods, demonstrating that the MOMC approaches are generally the most efficient choice for a given reasonably small error tolerance.

\subsection{Shallow water equations}
As second test problem, we consider the frictionless Shallow Water Equations (SWE) (see \cite{boscheri2023}):
\begin{subequations}
	\begin{align*}
		\frac{\partial \eta}{\partial t} + \frac{\partial (hu)}{\partial x} &= 0\,,\\
		\frac{\partial (hu)}{\partial t} + \frac{\partial \left(hu^2\right)}{\partial x} &= -  gh\frac{\partial \eta}{\partial x}\,,\\
		\frac{\partial b}{\partial t} &= 0\,,
	\end{align*}
\end{subequations}
where $\eta$ is the free surface elevation, $b$ is the bottom bathymetry, $h=\eta - b$ is the water depth, $u$ is the water velocity, and $g$ is gravity acceleration.

We now write the governing system of equations in dimensionless form to highlight the presence of stiff terms. We introduce the dimensionless variables 
\[ 
x^* = \frac{x}{\bar L}\,, \quad t^* = \frac{t}{\bar T}\,, \quad \eta^* = \frac{\eta}{\bar H}\,, \quad h^* = \frac{h}{\bar H}\,, \quad b^* = \frac{b}{\bar H}\,, \quad u^* = \frac{u}{\bar U}\,, 
\]
where $\bar L$, $\bar T$, $\bar H$, and $\bar U=\bar L/\bar T$ are the characteristic length, time, depth, and velocity, respectively. Substitution of these rescaled variables into the governing system of equations, dropping the star superscripts to ease notation, yields the following dimensionless SWE:
\begin{subequations}
	\begin{align*}
		\frac{\partial \eta}{\partial t} + \frac{\partial (hu)}{\partial x} &= 0\,,\\
		\frac{\partial (hu)}{\partial t} + \frac{\partial \left(hu^2\right)}{\partial x} &= - \frac{h}{\mathrm{Fr}^2}\frac{\partial \eta}{\partial x}\,,\\
		\frac{\partial b}{\partial t} &= 0\,.
	\end{align*}
\end{subequations}
Here, $\mathrm{Fr} = \bar{U} / \sqrt{g \bar{H}}$ is the Froude number, which measures the ratio between the convective velocity and the pressure (acoustic-gravity) wave speed. When the Froude number is very small (i.e., the acoustic wave speed dominates over the advective dynamics of the water), the system becomes stiff, highlighting the intrinsic multiscale nature of the dynamics. Fully explicit numerical methods would then require an extremely large number of small time steps to track the acoustic waves while the fluid itself moves very slowly. This motivates the need for a semi-implicit/IMEX temporal discretization of the equations (for which the stiff components are treated implicitly and all the rest is left explicit).

For the test case, in a domain $\mathcal{D}=[0,30]$, we consider a double Gaussian distribution of the water depth, with uncertain input $z$ affecting their variance, a constant null initial velocity, and a flat bottom profile:
\begin{subequations}
	\begin{align*}
		h(x,0,z) &= 1 + 0.01 e^{-\frac{(x-10)^2}{2\sigma^2(z)}} + 0.01e^{-\frac{(x-20)^2}{2\sigma^2(z)}}\,,\\
		u(x,0) &= 0\,,\\
		b(x,0) &= 0\,,
	\end{align*}
\end{subequations}
being $\sigma(z)=(1+z)$, $z \sim \mathcal{U}(0,1)$. We set the final time of the simulation $t_{end} = 1.0$.

We solve the problem applying the proposed MOMC approach considering, as deterministic solvers, the following three IMEX Finite Volume methods with increasing order of accuracy. 

\begin{figure}[b!]
	\centering
	\includegraphics[width=0.48\textwidth]{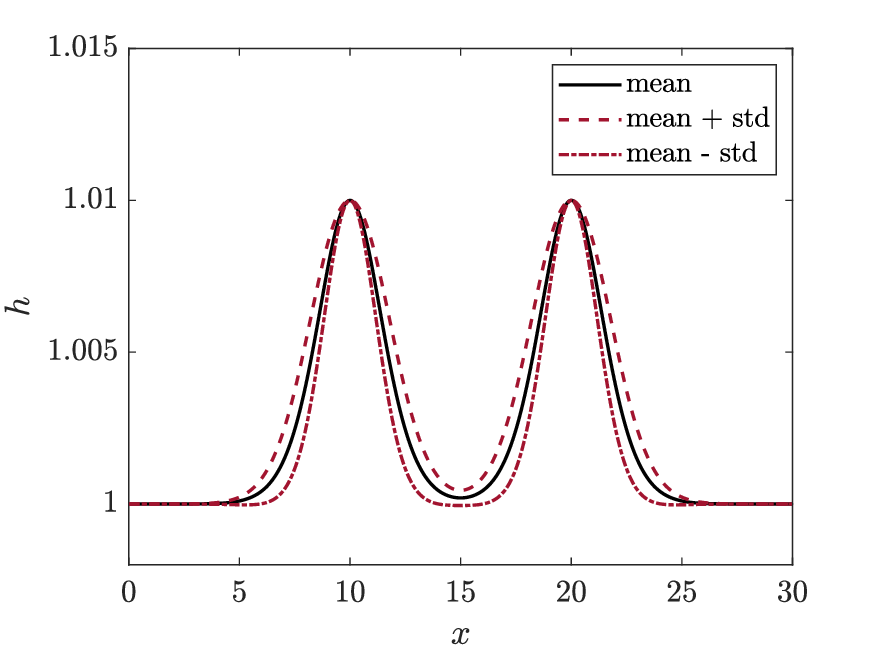}
	\includegraphics[width=0.48\textwidth]{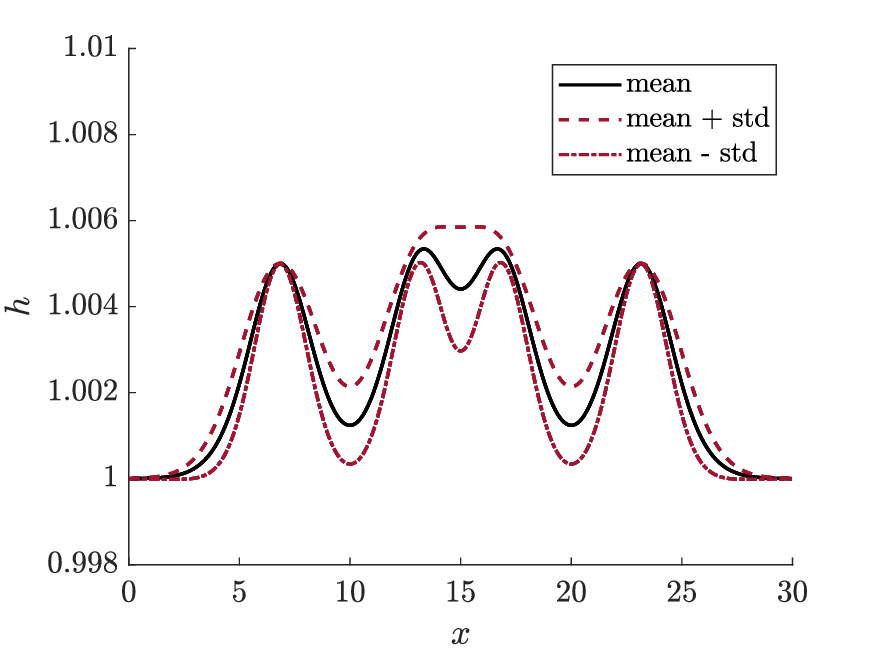}
	\caption{Shallow Water Equations. Reference solution of the water depth $h$ at the initial (left) and final (right) time of the simulation.}
	\label{test_SWE_sol}
\end{figure}
\begin{figure}[t!]
	\centering
	\includegraphics[width=0.48\textwidth]{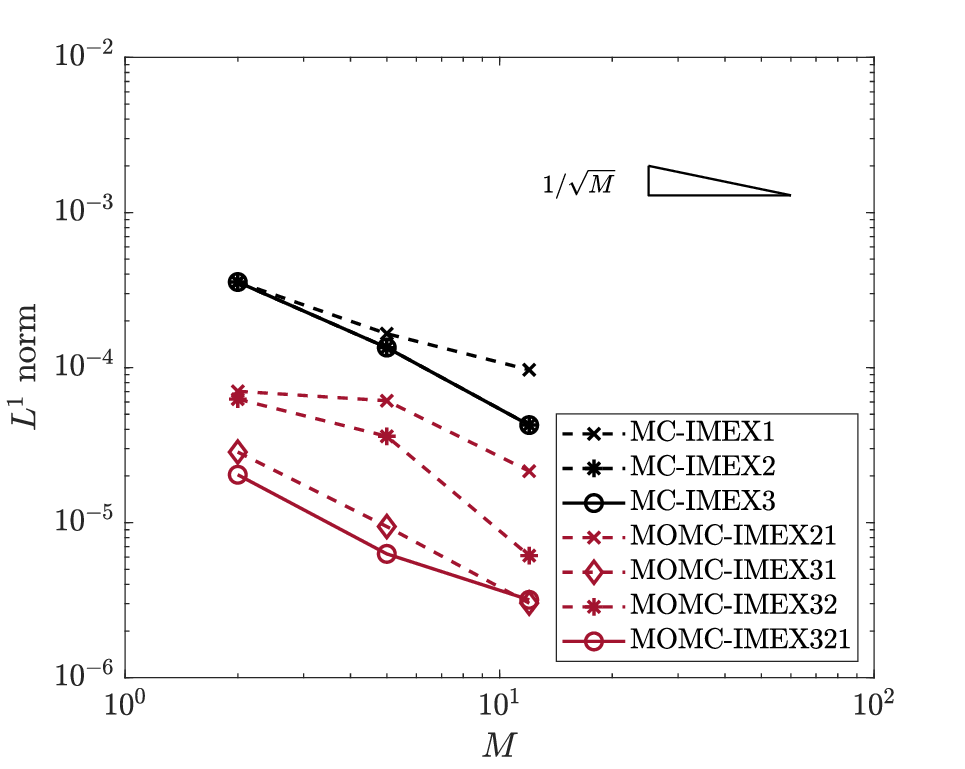}
	\includegraphics[width=0.48\textwidth]{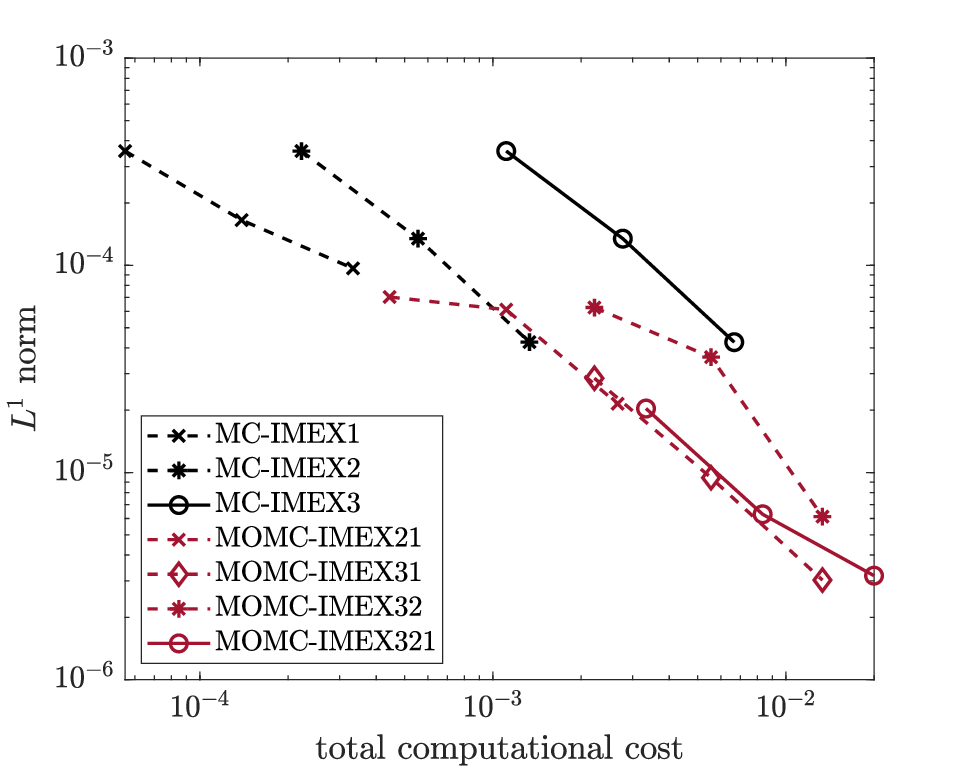}
	\caption{Shallow Water Equations. $L^1$-norm error in the expectation of the variable $h$ with respect to the number of samples $M$ used in the standard Monte Carlo (MC) methods and in the $L$-th level of the different Multi-Order MC (MOMC) methods used (left) and with respect to the total computational cost (right).}
	\label{test_SWE}
\end{figure}

\begin{enumerate}
	\item IMEX Runge–Kutta scheme of order $L=1$, namely ARS(1,1,1) \cite{ARS1997,IMEXbook} (which corresponds to a pair of forward and backward Euler), coupled with a first order Finite Volume space discretization and Rusanov numerical fluxes and Central Finite Difference discretization of the elliptic equation deriving from the resolution of the IMEX discretization of the system. The computational cost of a single run of the first order solver is $C_1 = 1\times1\times3=3$.
	
	\item L-stable IMEX Runge–Kutta scheme of order $L=2$, namely ARS(2,2,2) \cite{ARS1997,IMEXbook}, coupled with a second order Finite Volume space discretization based on a TVD reconstruction with minmod slope limiter and Rusanov numerical fluxes and Central Finite Difference discretization of the elliptic equation deriving from the resolution of the IMEX discretization of the system. The computational cost of a single run of the second order solver is $C_2 = 2\times2\times3=12$.
	
	\item L-stable Stiffly-Accurate IMEX Runge–Kutta scheme of order $L=3$, namely SI-IMEX(3,4,3) \cite{IMEXbook}, coupled with a third order Finite Volume discretization in space based on a WENO3 reconstruction of the boundary values and Rusanov numerical fluxes and a fourth order Central Finite Difference discretization of the elliptic equation deriving from the resolution of the IMEX discretization of the system. The computational cost of a single run of the third order solver is $C_3 = 4\times3\times5=60$.
\end{enumerate} 

The time step size is limited by the classical $\mathrm{CFL}$ stability condition, being $\mathrm{CFL} = 0.9$, without any restriction due to the Froude number thanks to the IMEX discretization. Finally, we discretize the physical domain with 500 computational cells.

The reader is invited to refer to \cite{IMEXbook} for further details on the IMEX numerical schemes.

\begin{remark}
	Note that the Finite Volume spatial discretizations chosen here and in the following tests represent only one possible option; alternative spatial discretizations could also be employed, since the key ingredient of the proposed methodology for multiscale systems lies in the IMEX temporal discretization.
\end{remark}

In Figure \ref{test_SWE_sol} we present a reference solution, showing the colliding Gaussian profiles, computed with the standard MC method with the chosen IMEX Runge–Kutta Finite Volume solver of order 3 using $M=300$ stochastic samples.

We compute the $L^1$-norm convergence curves of the solution using the same MOMC configurations discussed in the previous section for the inviscid Burgers' equation and compare the results with those obtained using the standard MC methods with different orders of accuracy of the deterministic solvers.

In Figure \ref{test_SWE}, we show the $L^1$-norm convergence curves for the expectation of the solution with respect to the number of samples $M$ used in the standard MC methods and at the $L$-th level of the different MOMC methods, along with convergence curves with respect to the total computational cost. As in the previous test, we observe that the various MOMC techniques offer advantages in terms of both accuracy and efficiency compared to the standard MC methods.

\subsection{Multiscale blood flow model}
As third application, we consider the multiscale blood flow model presented in \cite{bertaglia2023,bertaglia2020}. Averaging the frictionless incompressible Navier--Stokes equations over the vessel cross-section, while assuming axial symmetry in both the vessel and the flow, yields the following equations of conservation of mass and momentum augmented by the presence of the closing constitutive equation, which characterizes a viscoelastic behavior of the vessel wall \cite{bertaglia2023}:
\begin{subequations}
	\begin{align}
		\frac{\partial A}{\partial t} + \frac{\partial q}{\partial x} &= 0\,, \label{eq.contST}\\ 
		\frac{\partial q}{\partial t} + \frac{\partial}{\partial x}\left(\frac{q^2}{A}\right) + \dfrac{A}{\rho} \frac{\partial p}{\partial x} &= 0\,, \label{eq.momST}\\
		\frac{\partial p}{\partial t} + E_0 \,\mathcal{G}(A) \,\frac{\partial q}{\partial x} &= -\frac{1}{\tau_r}\left( p - p_{0} - E_{\infty}\,\mathcal{F}(A)\right) \,.\label{eq.PDE}
	\end{align}
	\label{completesyst}		
\end{subequations}
Under the assumption to treat only medium- to large-size arteries, we have \cite{bertaglia2020}
\be\nonumber
\mathcal{G}(A)=\frac{h_0 \sqrt{\pi }}{2A_0 \sqrt{A}},
\qquad
\mathcal{F}(A) =  \frac{h_0 \sqrt{\pi}}{A_0}\left( \sqrt{A} - \sqrt{A_0}\right).
\ee
Here, $A$ is the cross-sectional area of the vessel, $q$ is the averaged flow rate, $p$ is the averaged blood pressure, $\rho$ is the constant blood density. In Eq. \eqref{eq.PDE}, we can identify the three mechanical parameters that define the constitutive behavior of the vessel wall, namely: the instantaneous Young modulus $E_0$ (characterizing the instantaneous elastic response of the material when subject to load), the asymptotic Young modulus $E_\infty<E_0$ (characterizing the asymptotic response of the material subject to load, reached once the dissipation effects due to viscosity are over), and the relaxation time of the material $\tau_r$. The three parameters are linked together by the viscosity coefficient
$\eta =\tau_r E_0/(1 - \frac{E_{\infty}}{E_0})$.
Furthermore, we have the presence of the equilibrium (diastolic) pressure $p_0$, the equilibrium area $A_0$, and the constant wall thickness $h_0$.
Notice that all the variables of the system may depend on an uncertain input $z$ \cite{bertaglia2021a,colebank2024,fleeter2020,gao2020}.

The set of PDEs in \eqref{completesyst} constitutes a hyperbolic multiscale system with relaxation term, represented by the right hand side in Eq. \eqref{eq.PDE}. 
This type of mathematical models can describe phenomena characterized by different time scales (i.e., multiscale regimes) depending on the magnitude of the scaling parameters, leading to the resolution of problems that are notoriously complicated to approximate numerically due to the associated stiffness. 
To identify the terms responsible for the stiffness and to analyze the behavior of the model in the scaling (asymptotic) limits, it is convenient to write the governing equations in dimensionless form.
Hence, we fix the characteristic values for length ($\bar L$), time ($\bar T$), blood density ($\bar \rho$), cross-sectional vessel area ($\bar A$), vessel viscosity ($\bar \eta$), vessel elasticity ($\bar E = \bar \eta/\bar T$), and blood velocity ($\bar U = \bar L/\bar T$), to introduce the following dimensionless variables:
\[ 
x^* = \frac{x}{\bar L}\,, \quad t^* = \frac{t}{\bar T}\,, \quad \rho^* = \frac{\rho}{\bar \rho}\,, \quad A^* = \frac{A}{\bar A}\,, \quad q^* = \frac{q}{\bar A\bar U}\,, \quad p^* = \frac{p}{\bar \rho \bar U^2}\,,
\]
\[ 
A_0^* = \frac{A_0}{\bar A}\,, \quad p_0^* = \frac{p_0}{\bar \rho \bar U^2}\,, \quad E_0^* = \frac{E_0}{\bar E}\,, \quad E_{\infty}^* = \frac{E_{\infty}}{\bar E}\,, \quad \eta^* = \frac{\eta}{\bar \eta}\,, \quad\tau_r^* = \frac{\tau_r \bar E}{\bar \eta}\,.
\]
By omitting the star symbol of the rescaled variables for ease of reading, we finally obtain the following dimensionless form of system \eqref{completesyst}:
\begin{subequations}
	\begin{align}
		\frac{\partial A}{\partial t} + \frac{\partial q}{\partial x} &= 0\, , \label{eq.cont_adim}\\
		\frac{\partial q}{\partial t}+ \frac{\partial}{\partial x}\left(\frac{q^2}{A}\right) + \frac{A}{\rho} \, \frac{\partial p}{\partial x} &= 0 \,, \label{eq.mom_adim}\\
		\frac{\partial p}{\partial t} + \frac{E_0}{\mathsf{Re}} \mathcal{G}(A) \,\frac{\partial q}{\partial x} &= -\frac{1}{\tau_r}\left( p - p_0 - \frac{E_{\infty}}{\mathsf{Re}} \mathcal{F}(A)\right)\,, \label{eq.PDE_adim}
	\end{align}
	\label{completesyst_dimensionless}
\end{subequations}
where $\mathsf{Re} = \bar \rho \bar U \bar L / \bar \eta$ is the Reynolds number measuring the ratio between fluid inertia and viscous force of the wall (and not of the fluid, in contrast with the classical definition).

We can now observe that system \eqref{completesyst_dimensionless}, depending on the magnitude of $\tau_r$ and $E_0$, can describe different propagation dynamics. In particular, for the purpose of this work, we consider the following asymptotic behavior \cite{bertaglia2023}:
if $\tau_r\to0$ while $\eta\to0$ and $E_0$ remains finite, the vessel wall behaves in a purely elastic manner. Indeed, Eq. \eqref{eq.PDE_adim} relaxes toward the elastic tube law
$p = p_0 + \frac{E_{\infty}}{\mathsf{Re}} \mathcal{F}(A)$,
and the bio-fluid dynamics results governed by the following reduced-order hyperbolic system:
\begin{subequations}
	\begin{align}
		\frac{\partial A}{\partial t} + \frac{\partial q}{\partial x} &= 0 \,,\\
		\frac{\partial q}{\partial t} + \frac{\partial}{\partial x}\left(\frac{q^2}{A}\right) + \frac{A}{\rho} \frac{\partial}{\partial x}\left(p_0 + \frac{E_{\infty}}{\mathsf{Re}}\mathcal{F}(A)\right)&= 0 \,.
	\end{align}
	\label{completesyst_elastic}
\end{subequations}


Given the multiscale nature of the system, we solve the problems detailed in the following applying the AP-MOMC approach considering, as deterministic solvers, the following three AP-IMEX Finite Volume methods with increasing order of accuracy (the reader is referred to \cite{bertaglia2023} for further details on the numerical discretization of the system and to \cite{IMEXbook} for additional details on the IMEX numerical schemes).

\begin{enumerate}
	\item IMEX Runge Kutta scheme of order $L=1$, namely ARS(1,1,1) \cite{ARS1997,IMEXbook} (which corresponds to a pair of forward and backward Euler), coupled with a first order Finite Volume space discretization and Dumbser-Osher-Toro solver \cite{dumbser2011a} for the computation of the numerical fluxes. The computational cost of a single run of the first order solver is $C_1 = 1\times1=1$.
	
	\item L-stable IMEX Runge Kutta scheme of order $L=2$, namely ARS(2,2,2) \cite{ARS1997,IMEXbook}, coupled with a second order Finite Volume space discretization based on a TVD reconstruction with minmod slope limiter and Dumbser-Osher-Toro solver for the computation of the numerical fluxes. The computational cost of a single run of the second order solver is $C_2 = 2\times2=4$.
	
	\item L-stable Globally-Stiffly-Accurate IMEX Runge–Kutta scheme of order $L=3$, namely BPR(3,4,3) \cite{boscarino2017,IMEXbook}, coupled with a third order Finite Volume discretization in space based on a WENO3 reconstruction of the boundary values and Dumbser-Osher-Toro solver for the computation of the numerical fluxes. The computational cost of a single run of the third order solver is $C_3 = 4\times3=12$.
\end{enumerate} 

\subsubsection{Test 1: AP-MOMC}

\begin{figure}[b!]
	\centering
	\includegraphics[width=0.48\textwidth]{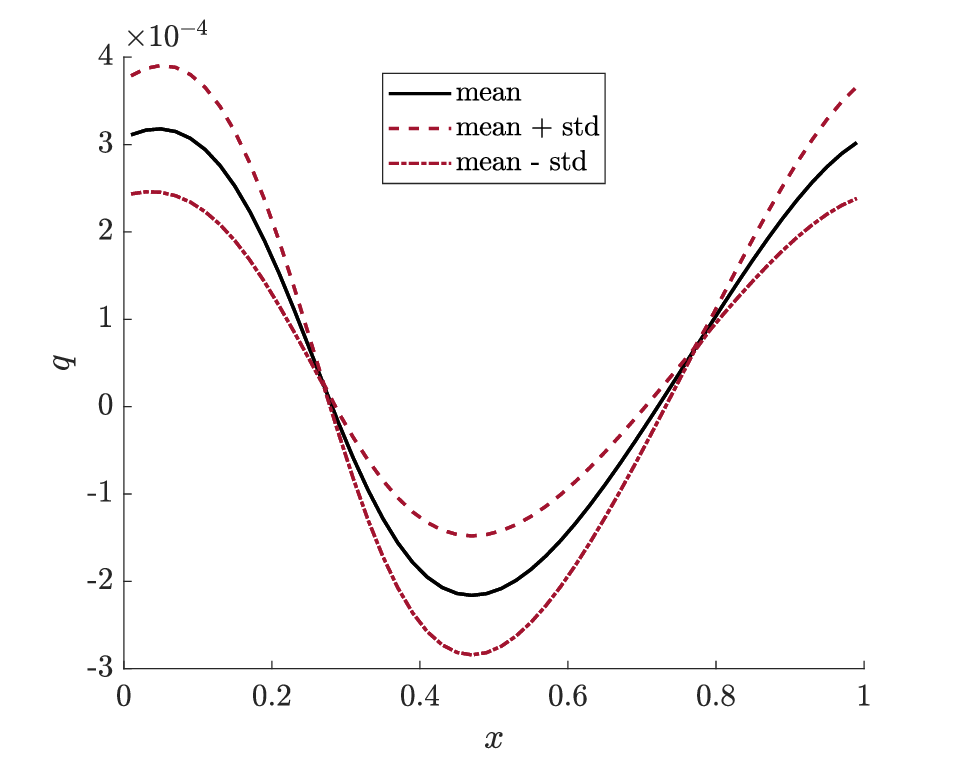}
	\includegraphics[width=0.48\textwidth]{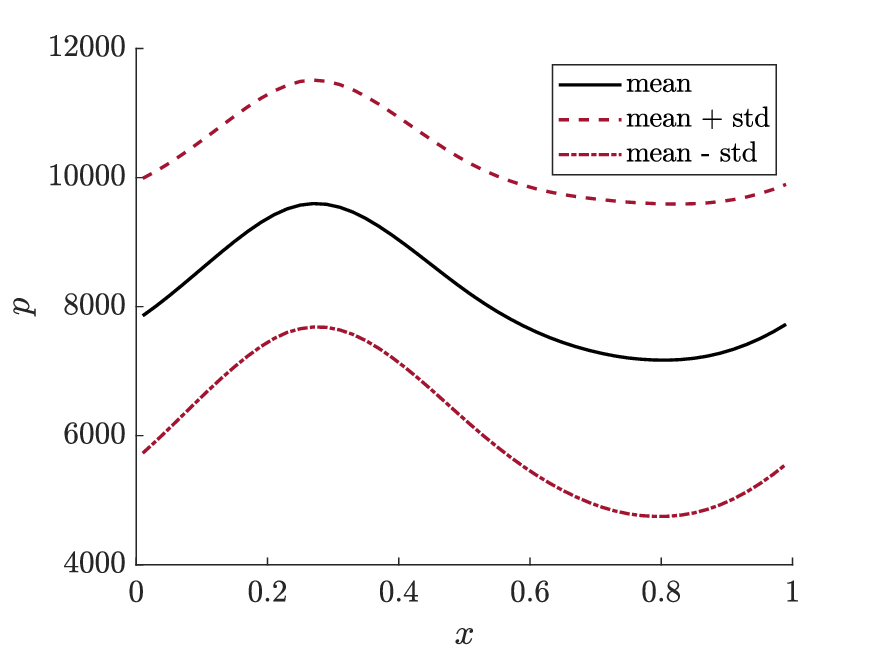}
	\caption{Blood flow model, Test 1. Reference solution at the final time of flow rate $q$ (left) and pressure $p$ (right).}
	\label{test1_BF_sol}
\end{figure}
\begin{figure}[b!]
	\centering
	\includegraphics[width=0.48\textwidth]{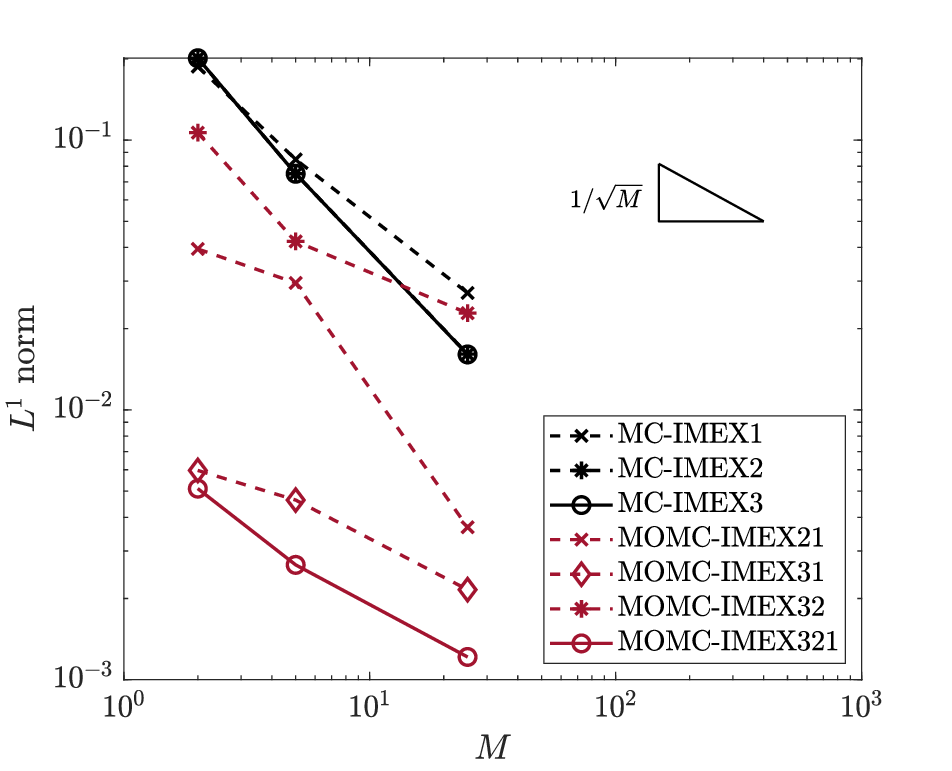}
	\includegraphics[width=0.48\textwidth]{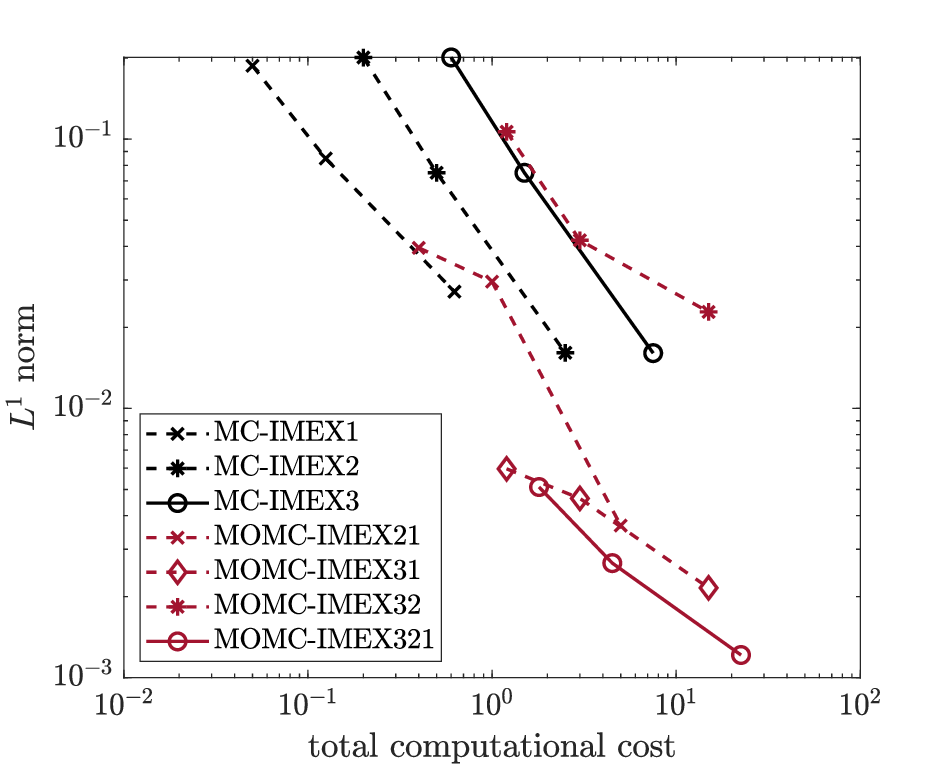}
	\caption{Blood flow model, Test 1. $L^1$-norm error in the expectation of the variable $p$ with respect to the number of samples $M$ used in the standard Monte Carlo (MC) methods and in the $L$-th level of the different Multi-Order MC (MOMC) methods used (left) and with respect to the total computational cost (right).}
	\label{test1_BF}
\end{figure}

In the first test case, we consider a sinusoidal initial distribution of the variables in the domain $\mathcal{D} = [0,1]$, 
\begin{subequations}
	\begin{align*}
		A(x,0) &= 0.0005 + 0.0001 \sin(2\pi x)\,,\\
		q(x,0) &= 0.00005 \,,\\
		p(x,0) &= 15000+ 5000 \sin(2\pi x)\,,\\
		A_0(x,0) &= 0.0005 + 0.0001 \sin(2\pi x)\,,\\
		p_0(x,0) &= 5000 + 500\cos(2\pi x)\,,\\
		E_0(x,0) &= 1000000 + 100000 \sin(2\pi x)\,,\\
		E_{\infty}(x,0) &= 800000 + 100000 \sin(2\pi x)\,,
	\end{align*}
\end{subequations}
accounting for an uncertain input $z$ affecting the viscosity coefficient of the vessel wall $\eta $ and, consequently, the relaxation time $\tau_r(z) = \frac{\eta(z)}{E_0} \left(1 - \frac{E_{\infty}}{E_0}\right)$, being $\eta(z) = 5\cdot10^{5}(1+z)$, $z \sim \mathcal{U}(-1,1)$. The wall thickness is set to be $h= 0.0015$ and the blood density $\rho = 1050$. We discretize the physical domain with 50 computational cells and a time step resulting from the classical CFL stability condition, not influenced by the smallness of the relaxation time thanks to the IMEX discretization, setting $\mathrm{CFL} = 0.9$. The final time of the simulation is $t_{end} = 0.1$.

In Figure \ref{test1_BF_sol} we show a reference solution for the variables $q$ and $p$ computed with the standard MC method with the chosen AP-IMEX Runge–Kutta Finite Volume solver of order 3 using $M=100$ stochastic samples.

In Figure \ref{test1_BF}, we compare the $L^1$-norm convergence curves for the expectation of the solution of $p$ obtained using the same MOMC configurations as in the previous tests with those obtained using the standard MC methods. The curves are shown with respect to the number of samples $M$ used in the standard MC methods and at the $L$-th level of the different MOMC methods, as well as with respect to the total computational cost. Once again, we observe that the various MOMC techniques demonstrate clear advantages over the standard MC methods.

\subsubsection{Test 2: Bi-fidelity AP-MOMC}
In the second test case, we consider the same initial distributions of the variables defined for Test 1, but this time accounting for an uncertain input $z$ affecting the initial distribution of cross-sectional areas as follows:
\begin{subequations}
	\begin{align*}
		A(x,0,z) &= 0.0005 + a(z) \sin(2\pi x)\,,\\
		A_0(x,0,z) &= 0.0005 + a(z) \sin(2\pi x)\,,
	\end{align*}
\end{subequations}
with $a(z) = 1+z$, $z\sim \mathcal{U}(-1,1)$. The viscosity coefficient, this time, is set to be $\eta = 5\cdot10^5$, while the rest of the parameters are left as in Test 1.

In Figure \ref{test2_BF_sol}, we present a reference solution for the variables $A$ and $p$, computed using the standard MC method with the chosen AP-IMEX Finite Volume solver of order 3 and $M = 100$ stochastic samples.

For this test case, we exploit the asymptotic limit \eqref{completesyst_elastic} of the governing system of equations to evaluate the performance of the bi-fidelity  AP-MOMC technique discussed in Section \ref{AP-MOMC}. Specifically, we construct a bi-fidelity AP-MOMC method by incorporating the low-order reduced elastic system, solved with the first order AP-IMEX scheme, as the most inexpensive surrogate of the full-order model \eqref{completesyst_dimensionless}, thereby building a hierarchy of four levels in total. \rev{Since the low-order system differs from the full-order model by containing one fewer equation, we assume that its theoretical computational cost is 2/3 of that of the full-order viscoelastic model. Under this assumption, the number of samples is selected according to Eq. \eqref{eq:samples}. Finally, to further investigate the behavior of the proposed methodology, we have also considered a three-level hierarchy obtained by omitting the second-order solution, as well as a simpler two-level hierarchy combining only the reduced-order model with the third-order scheme solutions.}

\begin{figure}[b!]
	\centering
	\includegraphics[width=0.48\textwidth]{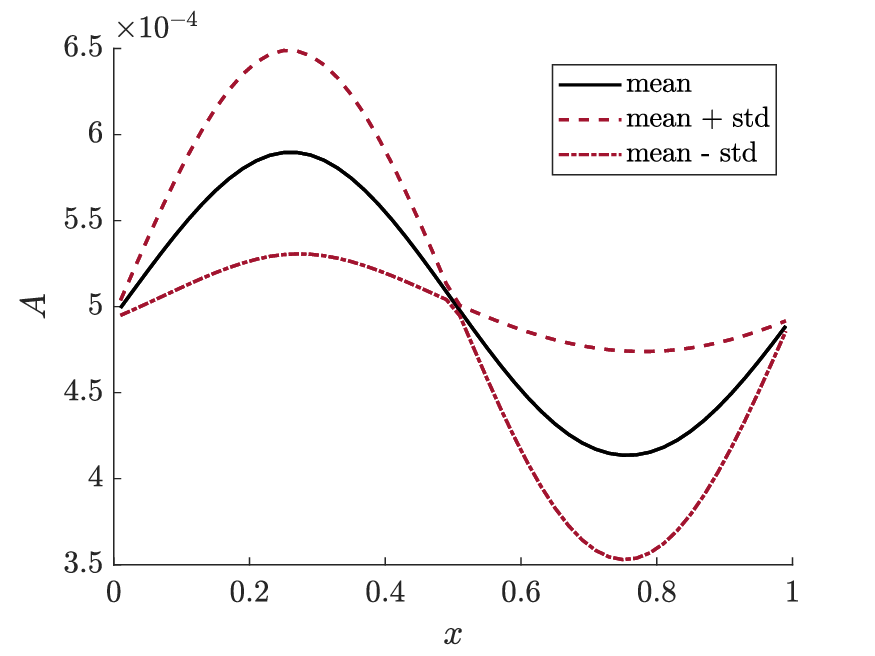}
	\includegraphics[width=0.48\textwidth]{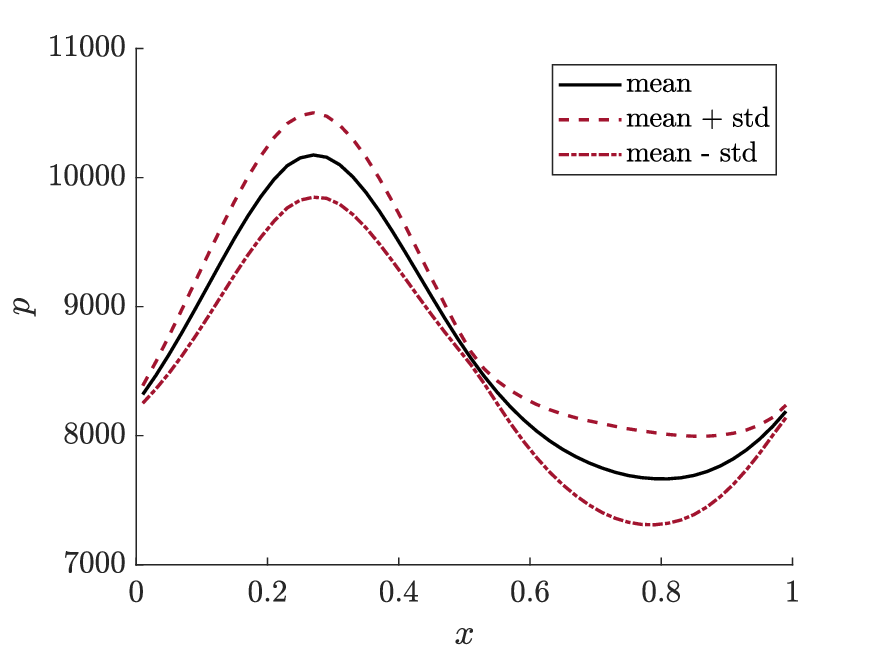}
	\caption{Blood flow model, Test 2. Reference solution at the final time of cross-sectional area $A$ (left) and pressure $p$ (right).}
	\label{test2_BF_sol}
\end{figure}
\begin{figure}[t!]
	\centering
	\includegraphics[width=0.49\textwidth]{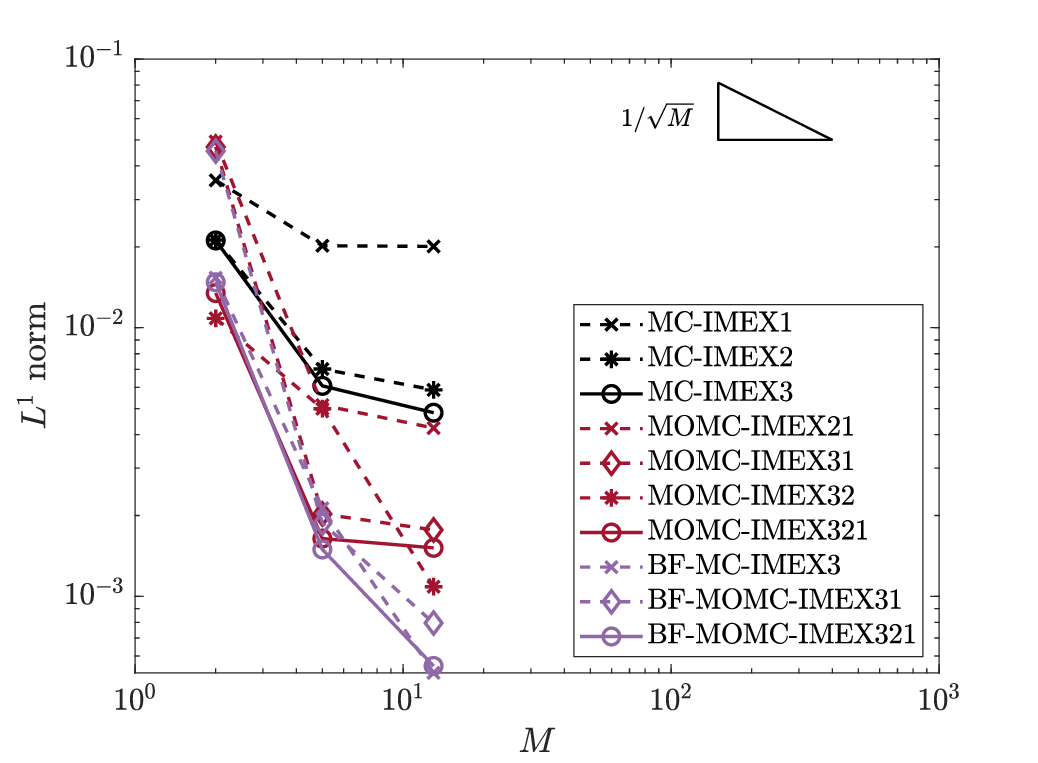}
	\includegraphics[width=0.48\textwidth]{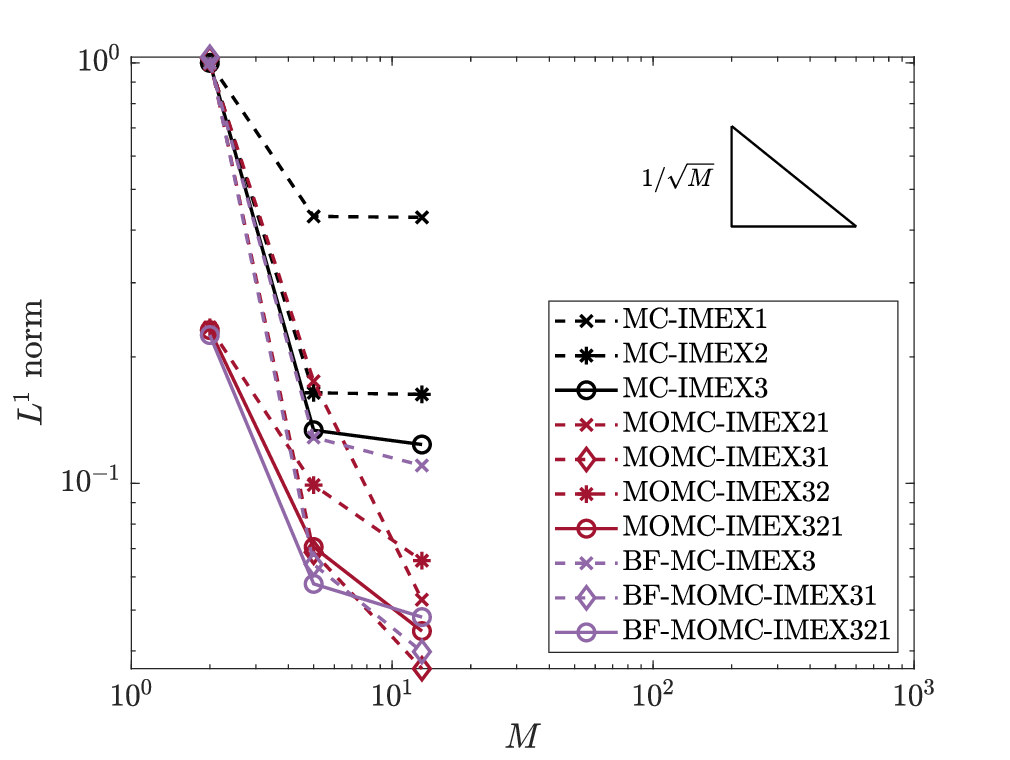}
	\caption{Blood flow model, Test 2. $L^1$-norm error in the expectation (left) and in the variance (right) of the variable $p$ with respect to the number of samples $M$ used in the standard Monte Carlo (MC) methods and in the $L$-th level of the different Multi-Order MC (MOMC) methods and Bi-Fidelity MOMC (BF-MOMC) used.}
	\label{test2_BF}
\end{figure}
\begin{figure}[t!]
	\centering
	\includegraphics[width=0.7\textwidth]{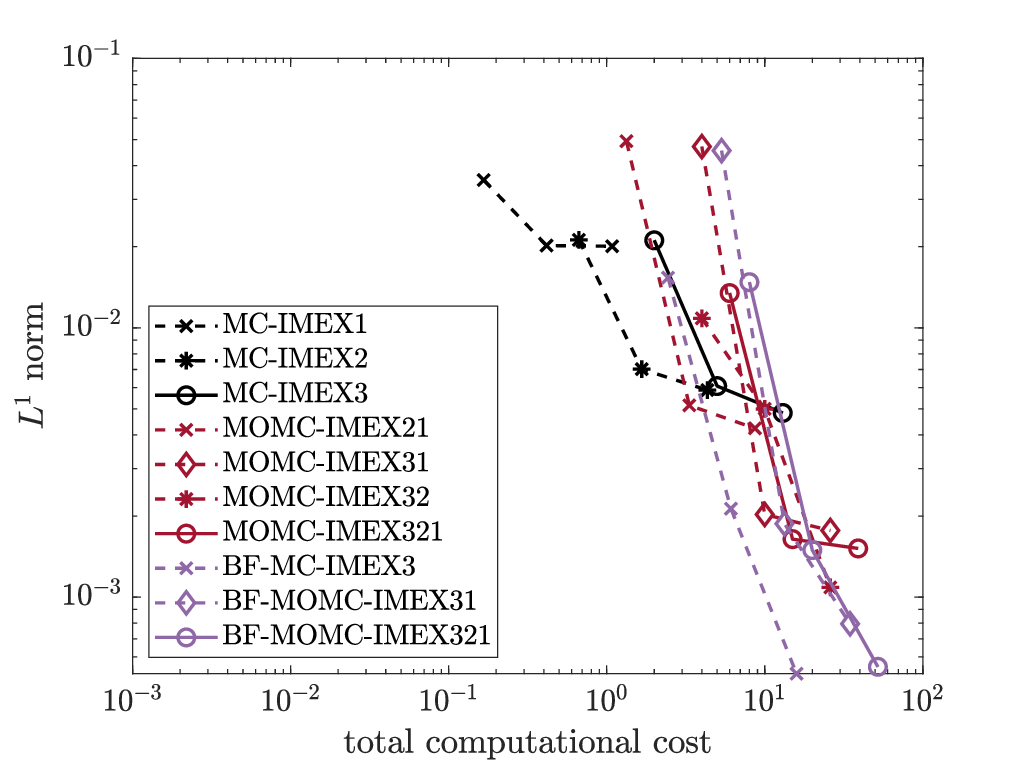}
	\caption{Blood flow model, Test 2. $L^1$-norm error in the expectation of the variable $p$ with respect to the total computational cost for the different Monte Carlo (MC), Multi-Order MC (MOMC) and Bi-Fidelity MOMC (BF-MOMC) methods used.}
	\label{test2_BF_tCPU}
\end{figure}

Results are shown in \rev{Figures \ref{test2_BF} and \ref{test2_BF_tCPU}} in terms of convergence curves for the expectation \rev{and variance} of the pressure wave. These plots highlight the improved performance achieved with the bi-fidelity versions of the MOMC by incorporating the low-order model within the hierarchical chain of levels of the standard MOMC approach, particularly when targeting small error tolerances.

\section{Conclusion}
In this work, we introduced a novel Multi-Order Monte Carlo (MOMC) framework for uncertainty quantification in multiscale systems governed by time-dependent PDEs. By constructing hierarchies from asymptotic-preserving (AP) IMEX Runge–Kutta finite volume schemes of varying orders, the method achieves significant variance reduction while preserving uniform accuracy across scales.

Unlike traditional Multi-Level Monte Carlo methods, which rely on nested meshes and can incur significant computational overhead, the MOMC strategy builds its hierarchy through different discretization orders, avoiding mesh refinements. This allows for seamless integration into complex computational domains.

We further extended the AP-MOMC framework with a bi-fidelity strategy, in which asymptotic model scalings are incorporated as additional levels in the hierarchy. This enrichment enables more efficient simulations of stiff multiscale systems while ensuring that critical asymptotic properties are preserved at the discrete level.

Various numerical experiments on the inviscid Burgers' equation, the shallow water equations, and a multiscale blood flow model demonstrate the effectiveness of the proposed methods. Across all test cases, the MOMC and AP-MOMC approaches consistently outperform standard Monte Carlo methods in terms of both accuracy and computational cost, particularly when targeting low error tolerances. The introduction of reduced-order models within the hierarchy, in the bi-fidelity AP-MOMC framework, further enhances performance without compromising the consistency or robustness of the solution.

As a direction for future work, we aim to extend the proposed methodology to multi-dimensional problems on unstructured grids, enabling its application to more general computational geometries and increasing its relevance for real-world simulations involving complex domains. In this context, we aim to further analyze the optimal selection of the number of samples at each level of the methodology, with the goal of improving the method’s overall effectiveness.
Aware of the limitations of the present proposal, which certainly requires further investigation to refine its individual components, we believe that the MOMC frameworks proposed here already provide powerful and versatile tools for uncertainty quantification in multiscale PDE systems, offering a promising foundation for future research and practical applications in complex engineering and biological settings.

\section*{Acknowledgments}
This work has been written within the activities of the GNCS group of INdAM (Italian National Institute of High Mathematics), whose support is acknowledged. It has also been partially supported by ICSC – Centro Nazionale di Ricerca in High Performance Computing, Big Data and Quantum Computing, funded by European Union–NextGenerationEU. WB and LP also acknowledge the financial support from the International Centre for Mathematical Sciences (Universities of Edinburgh and Heriot-Watt) under the program Research-In-Groups 2025.

%
\section*{Declarations}

\paragraph{Funding.} GB has been funded by the European Union–NextGenerationEU, MUR PRIN 2022 PNRR Project No. P2022JC95T. WB received financial support by the Italian Ministry of University (MUR) with the PRIN Project 2022 No. 2022N9BM3N. LP has been supported by the Royal Society under the Wolfson Fellowship “Uncertainty quantification, data-driven simulations and learning of multiscale complex systems governed by PDEs” and by MUR PRIN 2022 Project No. 2022KKJP4X. 

\paragraph{Conflicts of interest.} The authors declare that they have no conflict of interest.

\paragraph{Availability of data and material.} Data and material are available upon reasonable request addressed to the corresponding author.

\paragraph{Code availability.} The code is written in Matlab programming language and is available upon reasonable request addressed to the corresponding author.

\paragraph{Ethics approval.} Not applicable.

\paragraph{Consent to participate.} Not applicable.

\paragraph{Consent for publication.} Not applicable.

\section*{Data availability}
The datasets generated during the current study are available from the corresponding author upon reasonable request.


\bibliographystyle{abbrv}
\bibliography{library}

@article{geraci2015,
  title={A multifidelity control variate approach for the multilevel Monte Carlo technique},
  author={Geraci, G and Eldred, MS and Iaccarino, G},
  journal={Center for Turbulence Research Annual Research Briefs},
  pages={169--181},
  year={2015},
  publisher={Stanford University}
}

@article{ARS1997,
   author = {Uri M. Ascher and Steven J. Ruuth and Raymond J. Spiteri},
   doi = {10.1016/S0168-9274(97)00056-1},
   isbn = {01689274/97},
   issn = {01689274},
   issue = {2-3},
   journal = {Applied Numerical Mathematics},
   pages = {151-167},
   title = {Implicit-explicit {Runge-Kutta} methods for time-dependent partial differential equations},
   volume = {25},
   url = {https://linkinghub.elsevier.com/retrieve/pii/S0168927497000561},
   year = {1997}
}

@article{Lin2025,
   author = {Yiwen Lin and Liu Liu},
   journal = {Preprint ArXiv:2406.05489},
   title = {On a class of multi-fidelity methods for the semiclassical {Schrödinger} equation with uncertainties},
   url = {http://arxiv.org/abs/2406.05489},
   year = {2025}
}

@article{iacomini2025,
      title={Multi-fidelity and multi-level {Monte Carlo} methods for kinetic models of traffic flow}, 
      author={Elisa Iacomini and Lorenzo Pareschi},
      year={2025},
      eprint={2501.15967},
      url={https://arxiv.org/abs/2501.15967}, 
      journal={Preprint ArXiv:2501.15967},
}

@book{Rudin1987,
author = {Rudin, Walter},
title = {Real and complex analysis, 3rd ed.},
year = {1987},
isbn = {0070542341},
publisher = {McGraw-Hill, Inc.},
address = {USA}
}

@article{mouhot2006,
author = {Mouhot, Clément and Pareschi, Lorenzo},
title = {Fast Algorithms for Computing the Boltzmann Collision Operator},
journal = {Mathematics of Computation},
volume = {75}, 
issue = {256},
year = {2006}, 
pages = {1833--1852}
}

@article{russo2012,
author = {Russo, Giovanni and Santagati, Pietro and Yun, Seok-Bae},
title = {{Convergence of a Semi-Lagrangian Scheme for the BGK Model of the Boltzmann Equation}},
journal = {SIAM Journal on Numerical Analysis},
volume = {50},
number = {3},
pages = {1111-1135},
year = {2012},
doi = {10.1137/100800348},
URL = { https://doi.org/10.1137/100800348},
eprint = { https://doi.org/10.1137/100800348}
}

@InProceedings{Heinrich2001,
author="Heinrich, Stefan",
editor="Margenov, Svetozar
and Wa{\'{s}}niewski, Jerzy
and Yalamov, Plamen",
title="{Multilevel Monte Carlo Methods}",
booktitle="Large-Scale Scientific Computing",
year="2001",
publisher="Springer",
city="Berlin Heidelberg",
pages="58--67",
isbn="978-3-540-45346-8"
}

@article{Giles2008,
author = {Michael B. Giles},
title = {{Multilevel Monte Carlo Path Simulation}},
journal = {Operations Research},
volume = {56},
issue = {3},
pages = {607--617},
year = {2008}
}

@article{Barth2011,
author = {A. Barth, C. Schwab, N. Zollinger},
title = {{Multi-level Monte Carlo Finite Element} method for elliptic {PDEs} with stochastic coefficients},
journal = {Numerische Mathematik},
volume = {119},
pages = {123--161},
year = {2011},
doi = {10.1007/s00211-011-0377-0}
}

@article{menhorn2024,
      title={{Multilevel Monte Carlo} estimators for derivative-free optimization under uncertainty}, 
      author={Friedrich Menhorn and Gianluca Geraci and D. Thomas Seidl and Youssef M. Marzouk and Michael S. Eldred and Hans-Joachim Bungartz},
      year={2024},
	journal = {International Journal for Uncertainty Quantification},
	volume = {14},
	issue = {3},
	pages = {21--65},
      url={https://arxiv.org/abs/2305.03103}
}

@article{Motamed2020,
author = {M. Motamed}, 
title = {A {Multi-Fidelity Neural Network} Surrogate Sampling Method for Uncertainty Quantification},
journal = {International Journal for Uncertainty Quantification}, 
volume = {10}, 
issue = {6}, 
pages = {601--626}, 
year = {2020}, 
doi = {10.1615/Int.J.UncertaintyQuantification.2020033274}
}

@article{Beck2019,
title = {{IGA}-based multi-index stochastic collocation for random {PDEs} on arbitrary domains},
journal = {Computer Methods in Applied Mechanics and Engineering},
volume = {351},
pages = {330-350},
year = {2019},
issn = {0045-7825},
doi = {https://doi.org/10.1016/j.cma.2019.03.042},
url = {https://www.sciencedirect.com/science/article/pii/S0045782519301811},
author = {Joakim Beck and Lorenzo Tamellini and Raúl Tempone}
}

@article{Grote2022,
author = {Grote, Marcus J. and Michel, Simon and Nobile, Fabio},
title = {Uncertainty Quantification by {Multilevel Monte Carlo} and {Local Time-Stepping} for Wave Propagation},
journal = {SIAM/ASA Journal on Uncertainty Quantification},
volume = {10},
number = {4},
pages = {1601-1628},
year = {2022},
doi = {10.1137/21M1429047},
URL = {https://doi.org/10.1137/21M1429047},
eprint = {https://doi.org/10.1137/21M1429047}
}

@article{durrwachter2023,
      title={Data-integrated uncertainty quantification for the performance prediction of iced airfoils}, 
      author={Jakob Dürrwächter and Andrea Beck and Claus-Dieter Munz},
      year={2023},
      journal = {Preprint Ar{X}iv},
      eprint={2302.10294},
      archivePrefix={arXiv},
      primaryClass={physics.flu-dyn},
      url={https://arxiv.org/abs/2302.10294}, 
}

@article{Motamed2018,
   author = {Mohammad Motamed and Daniel Appelö},
   doi = {10.1137/16M1086388},
   issn = {0036-1429},
   issue = {1},
   journal = {SIAM Journal on Numerical Analysis},
   pages = {448-468},
   title = {A {MultiOrder Discontinuous Galerkin Monte Carlo} Method for Hyperbolic Problems with Stochastic Parameters},
   volume = {56},
   url = {https://epubs.siam.org/doi/10.1137/16M1086388},
   year = {2018}
}

@inproceedings{Sukys2013,
   author = {Jonas Šukys and Siddhartha Mishra and Christoph Schwab},
   city = {Berlin, Heidelberg},
   doi = {10.1007/978-3-642-41095-6_34},
   editor = {J. Dick and F. Kuo and G. Peters and I. Sloan},
   booktitle = {Monte Carlo and Quasi-Monte Carlo Methods 2012. Springer Proceedings in Mathematics \& Statistics},
   pages = {649-666},
   publisher = {Springer},
   title = {Multi-level Monte Carlo Finite Difference and Finite Volume Methods for Stochastic Linear Hyperbolic Systems},
   volume = {65},
   url = {https://link.springer.com/10.1007/978-3-642-41095-6_34},
   year = {2013}
}

@article{Blondeel2020,
   author = {Philippe Blondeel and Pieterjan Robbe and Cédric Van hoorickx and Stijn François and Geert Lombaert and Stefan Vandewalle},
   doi = {10.3390/a13050110},
   issn = {1999-4893},
   issue = {5},
   journal = {Algorithms},
   pages = {110},
   publisher = {MDPI AG},
   title = {{p-Refined Multilevel Quasi-Monte Carlo for Galerkin Finite Element Methods with Applications in Civil Engineering}},
   volume = {13},
   url = {https://www.mdpi.com/1999-4893/13/5/110},
   year = {2020}
}

@article{Dumbser2011a,
   author = {Michael Dumbser and Eleuterio F. Toro},
   doi = {10.1007/s10915-010-9400-3},
   issn = {08857474},
   journal = {Journal of Scientific Computing},
   pages = {70-88},
   title = {A simple extension of the {Osher} {Riemann} solver to non-conservative hyperbolic systems},
   volume = {48},
   year = {2011},
}

@article{Pareschi2005,
   author = {Lorenzo Pareschi and Giovanni Russo},
   doi = {10.1007/s10915-004-4636-4},
   isbn = {1573-7691},
   issn = {08857474},
   issue = {1/2},
   journal = {Journal of Scientific Computing},
   pages = {129-155},
   pmid = {12658535},
   title = {Implicit-explicit {Runge}-{Kutta} schemes and applications to hyperbolic systems with relaxation},
   volume = {25},
   year = {2005},
}

@book{Jin2017,
   author = {Shi Jin and Lorenzo Pareschi},
   isbn = {9783319671093},
   publisher = {Springer International Publishing},
   title = {Uncertainty Quantification for Hyperbolic and Kinetic Equations},
   year = {2017},
}

@book{Xiu2010,
   author = {Dongbin Xiu},
   city = {New Jersey},
   doi = {10.2307/j.ctv7h0skv},
   isbn = {9781400835348},
   publisher = {Princeton University Press},
   title = {Numerical Methods for Stochastic Computations},
   url = {http://www.jstor.org/stable/10.2307/j.ctv7h0skv},
   year = {2010},
}

@article{Chen2013,
   author = {Peng Chen and Alfio Quarteroni and Gianluigi Rozza},
   doi = {10.1002/cnm},
   isbn = {9789264086869},
   issn = {13979884},
   journal = {International Journal for Numerical Methods in Biomedical Engineering},
   pages = {698-721},
   pmid = {19861133},
   title = {Simulation-based uncertainty quantification of human arterial network hemodynamics},
   volume = {29},
   url = {http://knowledgebase.terrafrica.org/fileadmin/user_upload/terrafrica/docs/Final_Rockefeller_Report4April08.pdf},
   year = {2013},
}

@article{Fleeter2020,
   author = {Casey M. Fleeter and Gianluca Geraci and Daniele E. Schiavazzi and Andrew M. Kahn and Alison L. Marsden},
   doi = {10.1016/j.cma.2020.113030},
   issn = {00457825},
   journal = {Computer Methods in Applied Mechanics and Engineering},
   pages = {113030},
   title = {Multilevel and multifidelity uncertainty quantification for cardiovascular hemodynamics},
   volume = {365},
   url = {http://arxiv.org/abs/1908.04875 https://linkinghub.elsevier.com/retrieve/pii/S0045782520302140},
   year = {2020},
}

@article{Bertaglia2020,
   author = {Giulia Bertaglia and Valerio Caleffi and Alessandro Valiani},
   doi = {10.1016/j.cma.2019.112772},
   issn = {00457825},
   issue = {C},
   journal = {Computer Methods in Applied Mechanics and Engineering},
   pages = {112772},
   publisher = {Elsevier B.V.},
   title = {Modeling blood flow in viscoelastic vessels: the 1D augmented fluid–structure interaction system},
   volume = {360},
   url = {https://doi.org/10.1016/j.cma.2019.112772 https://linkinghub.elsevier.com/retrieve/pii/S0045782519306644},
   year = {2020},
}

@article{Boscarino2017,
   author = {Sebastiano Boscarino and Lorenzo Pareschi and Giovanni Russo},
   doi = {10.1137/M1111449},
   issn = {0036-1429},
   issue = {4},
   journal = {SIAM Journal on Numerical Analysis},
   pages = {2085-2109},
   title = {A Unified {IMEX} {Runge}--{Kutta} Approach for Hyperbolic Systems with Multiscale Relaxation},
   volume = {55},
   url = {https://epubs.siam.org/doi/10.1137/M1111449},
   year = {2017},
}

@article{Dimarco2014,
   author = {Giacomo Dimarco and Lorenzo Pareschi},
   doi = {10.1017/S0962492914000063},
   journal = {Acta Numerica},
   pages = {369-520},
   title = {Numerical methods for kinetic equations},
   volume = {23},
   year = {2014},
}

@article{Jin2015,
   author = {Shi Jin and Dongbin Xiu and Xueyu Zhu},
   doi = {10.1016/j.jcp.2015.02.023},
   issn = {10902716},
   journal = {Journal of Computational Physics},
   pages = {35-52},
   publisher = {Elsevier Inc.},
   title = {Asymptotic-preserving methods for hyperbolic and transport equations with random inputs and diffusive scalings},
   volume = {289},
   url = {http://dx.doi.org/10.1016/j.jcp.2015.02.023},
   year = {2015},
}

@article{Jin2016,
	author = {Jin, Shi and Xiu, Dongbin and Zhu, Xueyu},
	date = {2016/06/01},
	doi = {10.1007/s10915-015-0124-2},
	id = {Jin2016},
	isbn = {1573-7691},
	journal = {Journal of Scientific Computing},
	number = {3},
	pages = {1198--1218},
	title = {A Well-Balanced Stochastic {Galerkin} Method for Scalar Hyperbolic Balance Laws with Random Inputs},
	url = {https://doi.org/10.1007/s10915-015-0124-2},
	volume = {67},
	year = {2016}}

@inbook{Pareschi2020,
   author = {Lorenzo Pareschi},
   doi = {10.1007/978-3-030-67104-4\_5},
   editor = {G. Albi and S. Merino-Aceituno and A. Nota and M. Zanella},
   isbn = {978-3030671037},
   journal = {Trails in Kinetic Theory. SEMA SIMAI Springer Series, vol 25},
   pages = {141-181},
   publisher = {Springer},
   title = {An Introduction to Uncertainty Quantification for Kinetic Equations and Related Problems},
   url = {http://arxiv.org/abs/2004.05072 https://link.springer.com/10.1007/978-3-030-67104-4\_5},
   year = {2021},
}

@article{Bertaglia2021a,
   author = {Giulia Bertaglia and Valerio Caleffi and Lorenzo Pareschi and Alessandro Valiani},
   doi = {10.1016/j.jcp.2020.110102},
   issn = {00219991},
   journal = {Journal of Computational Physics},
   pages = {110102},
   title = {Uncertainty quantification of viscoelastic parameters in arterial hemodynamics with the {a-FSI} blood flow model},
   volume = {430},
   url = {http://arxiv.org/abs/2007.01907 https://linkinghub.elsevier.com/retrieve/pii/S0021999120308767},
   year = {2021},
}

@book{IMEXbook,
author = {Boscarino, Sebastiano and Pareschi, Lorenzo and Russo, Giovanni},
title = {Implicit-Explicit Methods for Evolutionary Partial Differential Equations},
publisher = {Society for Industrial and Applied Mathematics},
year = {2024},
doi = {10.1137/1.9781611978209},
address = {Philadelphia, PA},
URL = {https://epubs.siam.org/doi/abs/10.1137/1.9781611978209},
eprint = {https://epubs.siam.org/doi/pdf/10.1137/1.9781611978209}
}

@article{Dimarco2019,
   author = {Giacomo Dimarco and Lorenzo Pareschi},
   doi = {10.1016/j.jcp.2019.03.002},
   issn = {00219991},
   journal = {Journal of Computational Physics},
   pages = {63-89},
   title = {Multi-scale control variate methods for uncertainty quantification in kinetic equations},
   volume = {388},
   url = {https://linkinghub.elsevier.com/retrieve/pii/S0021999119301718},
   year = {2019},
}

@article{Gao2020,
   author = {Han Gao and Xueyu Zhu and Jian Xun Wang},
   doi = {10.1016/j.cma.2020.113047},
   issn = {00457825},
   journal = {Computer Methods in Applied Mechanics and Engineering},
   pages = {113047},
   publisher = {Elsevier B.V.},
   title = {A bi-fidelity surrogate modeling approach for uncertainty propagation in three-dimensional hemodynamic simulations},
   volume = {366},
   url = {https://doi.org/10.1016/j.cma.2020.113047},
   year = {2020},
}

@book{Pareschi2013,
   author = {Lorenzo Pareschi and Giuseppe Toscani},
   isbn = {0199655464},
   publisher = {Oxford University Press},
   title = {Interacting Multiagent Systems, Kinetic Equations And Monte Carlo Methods},
   year = {2013},
}

@article{Caflisch1998,
   author = {Russel E. Caflisch},
   doi = {10.1017/S0962492900002804},
   issn = {14740508},
   journal = {Acta Numerica},
   pages = {1-49},
   title = {{Monte Carlo and quasi-Monte Carlo methods}},
   volume = {7},
   year = {1998},
}

@article{Jin1995a,
   author = {Shi Jin and Zhouping Xin},
   doi = {10.1002/cpa.3160480303},
   issn = {10970312},
   issue = {3},
   journal = {Communications on Pure and Applied Mathematics},
   pages = {235-276},
   title = {The relaxation schemes for systems of conservation laws in arbitrary space dimensions},
   volume = {48},
   year = {1995},
}

@article{Dimarco2021,
   author = {Giacomo Dimarco and Liu Liu and Lorenzo Pareschi and Xueyu Zhu},
	journal = {Panoramas et Synthèses, Société Mathématique de France, to appear},
   title = {Multi-fidelity methods for uncertainty propagation in kinetic equations},
   url = {http://arxiv.org/abs/2112.00932},
   year = {2024},
}

@article{Jin2022,
   author = {Shi Jin},
   doi = {10.1017/S0962492922000010},
   issn = {0962-4929},
   journal = {Acta Numerica},
   pages = {415-489},
   title = {Asymptotic-Preserving Schemes for Multiscale Physical Problems},
   volume = {31},
   url = {https://www.cambridge.org/core/product/identifier/S0962492922000010/type/journal_article http://arxiv.org/abs/2112.05920},
   year = {2021},
}

@article{Boscheri2023,
   author = {Walter Boscheri and Maurizio Tavelli and Cristóbal E. Castro},
   doi = {10.1016/j.apnum.2022.11.022},
   issn = {01689274},
   journal = {Applied Numerical Mathematics},
   pages = {311-335},
   title = {An all {Froude} high order {IMEX} scheme for the shallow water equations on unstructured {Voronoi} meshes},
   volume = {185},
   url = {http://arxiv.org/abs/2209.00344 https://linkinghub.elsevier.com/retrieve/pii/S0168927422003099},
   year = {2023},
}

@article{Bertaglia2023,
   author = {Giulia Bertaglia and Lorenzo Pareschi},
   doi = {10.1137/23M1554230},
   issn = {1540-3459},
   issue = {3},
   journal = {Multiscale Modeling \& Simulation},
   pages = {1237-1267},
   title = {Multiscale Constitutive Framework of One-Dimensional Blood Flow Modeling: Asymptotic Limits and Numerical Methods},
   volume = {21},
   url = {https://epubs.siam.org/doi/10.1137/23M1554230},
   year = {2023},
}

@article{Pareschi2022,
   author = {Lorenzo Pareschi and Torsten Trimborn and Mattia Zanella},
   doi = {10.1615/Int.J.UncertaintyQuantification.2021037960},
   issn = {2152-5080},
   issue = {1},
   journal = {International Journal for Uncertainty Quantification},
   pages = {61-84},
   title = {MEAN-FIELD CONTROL VARIATE METHODS FOR KINETIC EQUATIONS WITH UNCERTAINTIES AND APPLICATIONS TO SOCIOECONOMIC SCIENCES},
   volume = {12},
   url = {http://arxiv.org/abs/2102.02589 http://www.dl.begellhouse.com/journals/52034eb04b657aea,5379079f419f0bb2,42876850121c88b0.html},
   year = {2022},
}

@article{Bertaglia2023a,
   author = {Giulia Bertaglia and Lorenzo Pareschi and Russel E. Caflisch},
   doi = {10.1007/s10915-024-02614-1},
   issn = {0885-7474},
   issue = {3},
   journal = {Journal of Scientific Computing},
   pages = {60},
   title = {Gradient-Based {Monte Carlo} Methods for Relaxation Approximations of Hyperbolic Conservation Laws},
   volume = {100},
   url = {https://link.springer.com/10.1007/s10915-024-02614-1},
   year = {2024},
}

@article{Mishra2012a,
   author = {S. Mishra and Ch. Schwab},
   doi = {10.1090/S0025-5718-2012-02574-9},
   issn = {0025-5718},
   issue = {280},
   journal = {Mathematics of Computation},
   pages = {1979-2018},
   title = {Sparse tensor multi-level Monte Carlo finite volume methods for hyperbolic conservation laws with random initial data},
   volume = {81},
   url = {https://www.ams.org/mcom/2012-81-280/S0025-5718-2012-02574-9/},
   year = {2012},
}

@article{Mishra2012,
   author = {S. Mishra and Ch Schwab and J. Šukys},
   doi = {10.1016/j.jcp.2012.01.011},
   issn = {10902716},
   issue = {8},
   journal = {Journal of Computational Physics},
   pages = {3365-3388},
   publisher = {Elsevier Inc.},
   title = {Multi-level {Monte Carlo} finite volume methods for nonlinear systems of conservation laws in multi-dimensions},
   volume = {231},
   url = {http://dx.doi.org/10.1016/j.jcp.2012.01.011},
   year = {2012},
}

@article{Colebank2024,
   author = {M. J. Colebank and N. C. Chesler},
   doi = {10.1007/s10237-024-01875-x},
   issn = {16177940},
   journal = {Biomechanics and Modeling in Mechanobiology},
   pmid = {39073691},
   publisher = {Springer Science and Business Media Deutschland GmbH},
   title = {Efficient uncertainty quantification in a spatially multiscale model of pulmonary arterial and venous hemodynamics},
   year = {2024},
}

@article{Peherstorfer2018,
   author = {Benjamin Peherstorfer and Karen Willcox and Max Gunzburger},
   doi = {10.1137/16M1082469},
   issn = {0036-1445},
   issue = {3},
   journal = {SIAM Review},
   pages = {550-591},
   publisher = {Society for Industrial and Applied Mathematics Publications},
   title = {Survey of Multifidelity Methods in Uncertainty Propagation, Inference, and Optimization},
   volume = {60},
   url = {https://epubs.siam.org/doi/10.1137/16M1082469},
   year = {2018},
}

@article{Peherstorfer2016,
   author = {Benjamin Peherstorfer and Karen Willcox and Max Gunzburger},
   doi = {10.1137/15M1046472},
   issn = {1064-8275},
   issue = {5},
   journal = {SIAM Journal on Scientific Computing},
   pages = {A3163-A3194},
   title = {Optimal Model Management for Multifidelity Monte Carlo Estimation},
   volume = {38},
   url = {http://epubs.siam.org/doi/10.1137/15M1046472},
   year = {2016},
}

@article{gorodetsky2020generalized,
  title={A generalized approximate control variate framework for multifidelity uncertainty quantification},
  author={Gorodetsky, Alex A and Geraci, Gianluca and Eldred, Michael S and Jakeman, John D},
  journal={Journal of Computational Physics},
  volume={408},
  pages={109257},
  year={2020},
  publisher={Elsevier}
}

@article{zanoni2024improved,
  title={Improved multifidelity Monte Carlo estimators based on normalizing flows and dimensionality reduction techniques},
  author={Zanoni, Andrea and Geraci, Gianluca and Salvador, Matteo and Menon, Karthik and Marsden, Alison L and Schiavazzi, Daniele E},
  journal={Computer methods in applied mechanics and engineering},
  volume={429},
  pages={117119},
  year={2024},
  publisher={Elsevier}
}

@inproceedings{geraci2017multifidelity,
  title={A multifidelity multilevel Monte Carlo method for uncertainty propagation in aerospace applications},
  author={Geraci, Gianluca and Eldred, Michael S and Iaccarino, Gianluca},
  booktitle={19th AIAA non-deterministic approaches conference},
  pages={1951},
  year={2017}
}

\end{document}